\documentclass[11pt,a4paper]{amsart}
\usepackage[margin=1.1in]{geometry}
\usepackage{amsmath,amssymb,amsthm,mathtools}
\usepackage{tikz}
\usepackage[colorlinks=true,linkcolor=blue!60!black,citecolor=blue!60!black,urlcolor=blue!60!black]{hyperref}
\usepackage{booktabs}

\newtheorem{theorem}{Theorem}[section]
\newtheorem{lemma}[theorem]{Lemma}
\newtheorem{corollary}[theorem]{Corollary}
\newtheorem{proposition}[theorem]{Proposition}
\newtheorem{conjecture}[theorem]{Conjecture}
\newtheorem{question}[theorem]{Question}
\theoremstyle{definition}
\newtheorem{observation}[theorem]{Observation}
\theoremstyle{remark}
\newtheorem{remark}[theorem]{Remark}

\DeclarePairedDelimiter\ceil{\lceil}{\rceil}
\DeclarePairedDelimiter\floor{\lfloor}{\rfloor}
\newcommand{\chie}{\chi_{=}}
\newcommand{\chies}{\chi_{=}^{*}}
\newcommand{\E}{\mathbb{E}}

\newcommand{\Prb}[1]{\mathbb{P}\left[#1\right]}
\newcommand{\IA}{I_{\mathrm{A}}}
\newcommand{\IM}{I_{\mathrm{M}}}
\newcommand{\IB}{I_{\mathrm{B}}}

\title{A square-root law for equitable coloring}

\author{Mohammad F. Marashdeh}
\address{Department of Mathematics, Mutah University, Karak, Jordan}
\email{marashdeh@mutah.edu.jo}

\date{\today}

\begin{document}
\raggedbottom

\begin{abstract}
An equitable $k$-coloring of a graph partitions its vertex set into $k$ independent sets whose sizes differ by at most one; the least such $k$ is the equitable chromatic number $\chie(G)$. Every known bound on $\chie$ valid for all graphs, beginning with the Hajnal--Szemer\'edi theorem, is linear in the maximum degree $\Delta$, and the star $K_{1,\Delta}$, for which $\chie=\ceil{\Delta/2}+1$, shows that no general bound below $\Delta/2$ exists. We prove that this obstruction is a shortage of vertices rather than an effect of the degree: every graph with $|V(G)|\ge3\chi(G)\Delta$ satisfies $\chie(G)=O\bigl(\chi(G)^{3/2}\sqrt{\Delta/\ln\Delta}\bigr)$ throughout the range $\chi(G)\le(\Delta/\ln\Delta)^{1/3}$, so that for graphs of large order the degree enters only through $\sqrt{\Delta/\ln\Delta}$, with the chromatic number governing the rest. For each fixed $\ell$, $\ell$-colorable graphs of sufficiently large order satisfy $\chie\le\bigl(2\sqrt2\,\ell\sqrt{\ell-1}+o(1)\bigr)\sqrt{\Delta/\ln\Delta}$, while a probabilistic construction supplies graphs of arbitrarily large order, bipartite when $\ell=2$, with $\chie\ge\tfrac13\sqrt{(\ell-1)\Delta/\ln\Delta}$: the order of growth $\Theta\bigl(\sqrt{\Delta/\ln\Delta}\bigr)$ is exact for every fixed chromatic number, and the extremal constant is determined up to a factor $O(\chi(G))$. All upper bounds are constructive, and a prescribed-anchor variant of the construction produces equitable colorings of bipartite graphs with $O(\sqrt\Delta)$ colors in optimal linear time.
\end{abstract}

\maketitle

\section{Introduction}\label{sec:intro}

All graphs are finite and simple. A proper $k$-coloring of a graph $G$ partitions $V(G)$ into $k$ independent sets, its \emph{color classes}; it is \emph{equitable} if the sizes of any two classes differ by at most one, equivalently if every class has size $\floor{n/k}$ or $\ceil{n/k}$, where $n:=|V(G)|$. The least $k$ for which $G$ admits an equitable $k$-coloring is the \emph{equitable chromatic number} $\chie(G)$. Because a graph that is equitably $k$-colorable need not be equitably $(k+1)$-colorable (for instance $K_{3,3}$, which is equitably $2$- but not $3$-colorable), one also studies the \emph{equitable chromatic threshold} $\chies(G)$, the least $k_0$ such that $G$ is equitably $k$-colorable for \emph{every} $k\ge k_0$. Always $\chie\le\chies$, and the two can differ by a factor of order $\Delta$: by Proposition~\ref{prop:threshsharp}, $K_{a,a}$ with $a$ odd has $\chie=2$ and $\chies=\Delta+1$. Which of the two a given result bounds therefore matters, and every statement below specifies it: Theorem~\ref{thm:upper} bounds $\chies$, and Theorems~\ref{thm:multi}, \ref{thm:uniform} and~\ref{thm:smallzeta} bound $\chie$.

Equitable colorings originate in a question of Erd\H{o}s~\cite{Erdos64}, answered by the theorem of Hajnal and Szemer\'edi~\cite{HajnalSzemeredi70}: every graph is equitably $k$-colorable whenever $k\ge\Delta(G)+1$, where $\Delta(G)$ denotes the maximum degree. Hence $\chies(G)\le\Delta(G)+1$ for every graph $G$, and $K_{\Delta,\Delta}$ with $\Delta$ odd shows this is sharp even among bipartite graphs (Proposition~\ref{prop:threshsharp}). Chen, Lih and Wu~\cite{ChenLihWu94} conjectured that $\Delta$ colors suffice for every graph containing neither $K_{\Delta+1}$ nor, for odd $\Delta$, the complete bipartite graph $K_{\Delta,\Delta}$; Lih and Wu~\cite{LihWu96} confirmed this for bipartite graphs. For trees, Chen and Lih~\cite{ChenLih94} determined the equitable chromatic number exactly, and Bollob\'as and Guy~\cite{BollobasGuy83} showed that almost all trees are equitably $3$-colorable, an early indication that few colors suffice for sparse graphs with many vertices.

A different line of work asks for colorings of a bipartite graph whose classes are \emph{balanced}, meeting the two sides equally. Feige and Kogan~\cite{FeigeKogan10} introduced the problem; Chakraborti~\cite{Chakraborti23}, sharpening Axenovich, Sereni, Snyder and Weber~\cite{ASSW21}, proved that, for every $\varepsilon>0$, a balanced bipartite graph with maximum degree $\Delta$ and sufficiently many vertices on each side has a balanced independent set with $(1-\varepsilon)(\ln\Delta/\Delta)n$ vertices on each side, and can be partitioned into $(1+\varepsilon)\Delta/\ln\Delta$ balanced independent sets. That is the closest precedent for the phenomenon studied here; the two problems differ as follows. Balance forces every class to meet the small side, and the domination bound on balanced independent sets then makes $\Theta(\Delta/\ln\Delta)$ classes necessary. Equitability imposes no such constraint: a class may lie wholly inside one side, and the correct order is $\Theta(\sqrt{\Delta/\ln\Delta})$, a different power of $\Delta$. The obstruction behind our lower bound is nevertheless the same domination phenomenon, applied to a bipartition that is as unbalanced as equitability permits. For sparse graphs the analogous question is asked in terms of degeneracy rather than order; see Kostochka and Nakprasit~\cite{KostochkaNakprasit03}.

All of these bounds are linear in $\Delta$, or within a logarithm of it, and the literature offers no reason to expect anything smaller: the star $K_{1,\Delta}$ satisfies $\chie(K_{1,\Delta})=\ceil{\Delta/2}+1$, so no bound better than $\Delta/2$ can hold even among trees.

The star, however, is misleading: its obstruction arises from a shortage of vertices, not from adjacency. With $n=\Delta+1$, every color class lies within a single part, and the larger part alone necessitates $\ceil{\Delta/2}$ classes. This phenomenon must be distinguished from a related but distinct one. In $K_{\Delta,\Delta}$, every class is also confined to one part, yet $\chie=2$, the two parts themselves forming equal-sized classes; the large value pertains to $\chies$, not $\chie$, and stems from arithmetic rather than from a shortage of room (Proposition~\ref{prop:threshsharp}). The star is the extremal example for $\chie$, and it is $\chie$ that is the focus of the present work. The star is \emph{small} relative to its degree, so the two roles of $\Delta$ --- constraining which vertices may share a class, and determining how many classes are needed at all --- are conflated. For graphs with $n\gg\Delta$ these roles separate; the present paper determines the behavior in precisely that regime.

\begin{question}\label{q:main}
Fix $\zeta>1$ and consider bipartite graphs $G$ with $|V(G)|\ge\zeta\Delta(G)$ and $\Delta(G)$ large. Must $\chie(G)$ grow linearly in $\Delta$, as it does for the star, or is it of smaller order, or perhaps even bounded?
\end{question}

Although both alternatives are a priori plausible, neither is correct. A linear answer would make the star a symptom rather than an anomaly, with $\Delta$ governing $\chie$ however much room is available; a bounded answer would mean that with enough room the degree ceases to matter. The correct growth rate lies strictly between them and is one that neither alternative suggests: it is $\sqrt\Delta$, to within a factor $O(\sqrt{\ln\Delta})$. Bipartiteness, moreover, plays no role in the phenomenon: the same law holds whenever the chromatic number is bounded, and fails precisely when it is not. For a class specified by order, maximum degree and chromatic number alone, no bound sublinear in $\Delta$ appears to have been known. The results that follow determine the extremal growth rate up to a constant factor for every fixed chromatic number, identify the chromatic number as the parameter that governs it, and realize a bound of $O(\sqrt\Delta)$ colors by an algorithm of optimal running time.

\begin{theorem}[Adaptive anchors]\label{thm:multi}
Let $\ell\ge2$. For $u>1$ put $h(u):=u\ln\frac{u}{u-1}$, and for $s>1$ put
\begin{equation}\label{eq:Gamma}
c(s)^2\;:=\;2\inf_{1/s<v<1}\frac{h(sv)}{v(1-v)},
\qquad
\Gamma(\ell,\zeta)\;:=\;\ell\sqrt{\ell-1}\;c(\zeta/\ell) .
\end{equation}
For every real $\zeta>\ell$ and every real $\gamma>\Gamma(\ell,\zeta)$ there is $\Delta_0=\Delta_0(\ell,\zeta,\gamma)$ such that every $\ell$-colorable graph $G$ with $\Delta:=\Delta(G)\ge\Delta_0$ and $|V(G)|\ge\zeta\Delta$ satisfies
\[
\chie(G)\;\le\;\ceil*{\gamma\sqrt{\frac{\Delta}{\ln\Delta}}\,}.
\]
\end{theorem}

The constant depends on $\zeta$ only through the ratio $\zeta/\ell$; the governing quantity is the order per unit of chromatic number and degree.

\begin{remark}\label{rem:cfunction}
Since $h$ is decreasing with $h(u)\downarrow1$ as $u\to\infty$ --- both elementary: $h'(u)=\ln\tfrac{u}{u-1}-\tfrac1{u-1}<0$ by $\ln(1+x)<x$, and $h(u)=u\ln\bigl(1+\tfrac1{u-1}\bigr)\to1$ --- the function $c$ is decreasing, $c(s)\le\sqrt{8h(s/2)}$ for $s>2$ (take $v=\tfrac12$), and $c(s)\downarrow2\sqrt2$ as $s\to\infty$, while $c(s)\to\infty$ as $s\downarrow1$. Numerically $c(2)=4.076$, $c(3)=3.487$, $c(4)=3.277$, $c(8)=3.026$ and $c(25)=2.887$. For bipartite graphs this gives $\Gamma(2,4)=8.152$, $\Gamma(2,8)=6.554$, $\Gamma(2,25)=5.899$, and $\Gamma(2,\zeta)\downarrow4\sqrt2=5.657$. Since $c(s)^2$ is an infimum, any test point certifies an upper bound, and the values above are certified to the digits shown; for instance $v=0.574$ gives $c(4)^2\le2h(2.296)/(0.574\cdot0.426)<10.74$, hence $\Gamma(2,8)=2c(4)<6.56<6.6$, as Corollary~\ref{cor:bip} requires.
\end{remark}

Taking $\ell=2$ and $\zeta=8$ (Remark~\ref{rem:cfunction}) gives the simplest instance.

\begin{corollary}\label{cor:bip}
There exists $\Delta_0$ such that every bipartite graph $G$ with $\Delta:=\Delta(G)\ge\Delta_0$ and $|V(G)|\ge8\Delta$ satisfies $\chie(G)\le\ceil*{6.6\sqrt{\Delta/\ln\Delta}\,}$.
\end{corollary}

The proof does not require $\ell$ to be fixed, and tracking the dependence yields a bound in which the chromatic number appears as a parameter rather than a hypothesis.

\begin{theorem}[Chromatic number as a parameter]\label{thm:uniform}
There is an absolute constant $\Delta_1$ such that every graph $G$ with $\Delta:=\Delta(G)\ge\Delta_1$,
\[
\ell\;:=\;\chi(G)\;\le\;\Bigl(\frac{\Delta}{\ln\Delta}\Bigr)^{1/3}\qquad\text{and}\qquad |V(G)|\;\ge\;3\,\ell\,\Delta
\]
satisfies
\[
\chie(G)\;\le\;\ceil*{5\,\ell^{3/2}\sqrt{\frac{\Delta}{\ln\Delta}}\,}\;=\;O\!\left(\chi(G)^{3/2}\sqrt{\frac{\Delta}{\ln\Delta}}\right).
\]
\end{theorem}

The restriction on $\chi(G)$ is close to the point beyond which the bound ceases to say anything: its right-hand side exceeds $\Delta$ once $\chi(G)>(\tfrac15)^{2/3}(\Delta\ln\Delta)^{1/3}$, where the Hajnal--Szemer\'edi bound $\chie\le\Delta+1$ is already better. The two thresholds differ by a factor $\Theta(\ln^{2/3}\Delta)$, so Theorem~\ref{thm:uniform} covers all but a logarithmically narrow portion of the range on which \emph{this} bound is meaningful. This is a statement about the bound, not about $\chie$: we do not claim that no sublinear bound holds for larger $\chi(G)$, and we know of no obstruction to one. What Theorem~\ref{thm:uniform} does establish is that $\chi$ --- not bipartiteness, and not boundedness of $\chi$ --- is the parameter that governs the answer.

The next result shows that both the exponent $\tfrac12$ and the logarithm are correct.

\begin{theorem}[Lower bound]\label{thm:lower}
Let $\ell\ge2$. For every real $\zeta>1$ there exists $\Delta_0=\Delta_0(\ell,\zeta)$ such that for every integer $\Delta\ge\Delta_0$ there is a graph $G$ with $\chi(G)\le\ell$, $\Delta(G)=\Delta$ and $|V(G)|\ge\zeta\Delta$ for which
\[
\chie(G)\;\ge\;\frac{1}{3}\sqrt{\frac{(\ell-1)\Delta}{\ln\Delta}}\,.
\]
\end{theorem}

For $\ell=2$ this is a bipartite construction and matches Theorem~\ref{thm:multi} up to an absolute constant. For larger $\ell$ it shows that the growth in $\ell$ is at least $\sqrt{\ell-1}$, against the $\ell^{3/2}$ of Theorem~\ref{thm:multi}; see Section~\ref{sec:concl}.

Together, Theorems~\ref{thm:multi} and~\ref{thm:lower} determine the extremal function
\[
f_{\ell,\zeta}(\Delta)\;:=\;\max\bigl\{\chie(G)\ :\ \chi(G)\le\ell,\ \Delta(G)=\Delta,\ |V(G)|\ge\zeta\Delta\bigr\}
\]
up to a multiplicative factor $O(\ell)$, for every fixed $\ell$ and every $\zeta>\ell$:
\[
\frac13\sqrt{\frac{(\ell-1)\Delta}{\ln\Delta}}\;\le\;f_{\ell,\zeta}(\Delta)\;\le\;\bigl(\Gamma(\ell,\zeta)+o(1)\bigr)\sqrt{\frac{\Delta}{\ln\Delta}} ,
\qquad
\Gamma(\ell,\zeta)\;\downarrow\;2\sqrt2\,\ell\sqrt{\ell-1}\ \text{ as }\zeta/\ell\to\infty .
\]
In the limit $\zeta/\ell\to\infty$ the upper and lower constants differ by a factor $6\sqrt2\,\ell$, which is below $17$ when $\ell=2$. This answers Question~\ref{q:main}. The order $\sqrt{\Delta/\ln\Delta}$ is a genuine intermediate: $\chie$ is neither linear in $\Delta$ nor bounded. The lower bound is independent of $\zeta$, and the upper bound depends on it only through a constant that stabilizes once $\zeta/\ell$ is large.

The anchor sets of Theorem~\ref{thm:multi} are \emph{found}, by Lemma~\ref{lem:avg}, and Lemma~\ref{lem:derand} locates them deterministically but not in linear time. Prescribing them instead costs a factor $\sqrt{\ln\Delta}$ and yields an optimal algorithm; we record this trade-off because it is the only route to such an algorithm known to us.

\begin{theorem}[Prescribed anchors]\label{thm:upper}
Let $\ell\ge2$. For every real $\zeta>2\ell$ and every real $\gamma>2\ell\sqrt{\dfrac{\zeta(\ell-1)}{\zeta-2\ell}}$ there exists $\Delta_0=\Delta_0(\ell,\zeta,\gamma)$ such that every $\ell$-colorable graph $G$ with $\Delta:=\Delta(G)\ge\Delta_0$ and $|V(G)|\ge\zeta\Delta$ satisfies
\[
\chies(G)\;\le\;\ceil*{\gamma\sqrt{\Delta}\,},
\]
that is, $G$ is equitably $k$-colorable for every $k\ge\ceil{\gamma\sqrt\Delta}$. For bipartite graphs this reads: $\zeta>4$ and $\gamma>4\sqrt{\zeta/(\zeta-4)}$.
\end{theorem}

Theorem~\ref{thm:upper} also requires more room than Theorem~\ref{thm:multi}: $\zeta>2\ell$ against $\zeta>\ell$. Unlike Theorem~\ref{thm:multi}, its anchor sets are fixed in advance, computed from $\ell$, $\zeta$, $\gamma$ and the part sizes alone; the linear-time algorithm requires this, and it is also the reason the bound holds for every $k\ge\ceil{\gamma\sqrt\Delta}$ rather than for one value of $k$.

Theorem~\ref{thm:upper} needs $n>2\ell\Delta$. For bipartite graphs that hypothesis can be reduced to the best possible one.

\begin{theorem}[Little room]\label{thm:smallzeta}
For every real $\zeta>1$ there are $C(\zeta)$ and $\Delta_0(\zeta)$ such that every bipartite graph $G$ with $\Delta:=\Delta(G)\ge\Delta_0$ and $|V(G)|\ge\zeta\Delta$ satisfies $\chie(G)\le C(\zeta)\sqrt\Delta$. One may take $C(\zeta)=\max\{4\sqrt{\zeta/(\zeta-1)},2\sqrt{10}\}+o(1)$, the $o(1)$ as $\Delta\to\infty$.
\end{theorem}

The hypothesis $\zeta>1$ cannot be weakened, since the star has $n=\Delta+1$ and $\chie=\ceil{\Delta/2}+1$; so for $\chie$ the order threshold is exactly $n>\Delta$. For $\chies$ it is not: Proposition~\ref{prop:threshsharp} exhibits bipartite graphs with $n=2\Delta$ and $\chies=\Delta+1$, so the hypothesis $\zeta>2\ell$ of Theorem~\ref{thm:upper} is within a factor $2$ of best possible when $\ell=2$.

\begin{remark}\label{rem:supersedes}
Theorem~\ref{thm:upper} is proved from first principles and does not depend on any previous bound for this class. It also subsumes the linear bounds known for this class. A bound $\chie(G)\le\ceil{\Delta/2}+1$ for bipartite graphs with $n\ge\zeta\Delta$ and $\Delta$ large was obtained for $\zeta>\tfrac{41}{2}$ by Nikabadi~\cite{Nikabadi26}; Theorem~\ref{thm:upper} applies on the wider range $\zeta>4$; on the common range $\zeta>\tfrac{41}{2}$, for $\gamma$ close to its critical value $4\sqrt{\zeta/(\zeta-4)}$, it gives $\ceil{\gamma\sqrt\Delta}\le\ceil{\Delta/2}+1$ for every $\Delta\ge75$; on the range where both apply it is therefore smaller by a factor $\Theta(\sqrt\Delta)$. (The comparison is of the two bounds; each still requires its own lower bound on $\Delta$.)
\end{remark}

A bound on $\chie$ does not by itself bound $\chies$; Theorem~\ref{thm:upper} is stated for the threshold because its construction, unlike that of Theorem~\ref{thm:multi}, is monotone in $k$ once the number of reserved classes is fixed. Its proof is constructive: it exhibits the coloring rather than merely certifying its existence, and the construction is a single greedy pass, so it runs in optimal time.

\begin{theorem}[Linear time]\label{thm:alg}
There is an algorithm which, given a bipartite graph $G$ (as adjacency lists) satisfying the hypotheses of Theorem~\textup{\ref{thm:upper}} with parameters $\zeta$ and $\gamma$, and any $k$ with $\ceil{\gamma\sqrt\Delta}\le k\le\Delta$, outputs an equitable $k$-coloring of $G$ in time $O(n+m)$, where $m:=|E(G)|$. Given a proper $\ell$-coloring as part of the input, the same holds for $\ell$-colorable graphs.
\end{theorem}

This matches the $\Omega(n+m)$ cost of reading the input. We know of no previous linear-time algorithm producing an equitable coloring of a bipartite graph with a sublinear number of colors.

\subsection*{Method}
The order of magnitude is captured by the following heuristic, stated for a bipartite graph with parts $A$ and $B$. With $k$ colors, a constant fraction of $B$ may be reserved for classes that mix the two sides. Each such class is anchored at a few vertices of $A$, and anchoring one vertex forbids its $\Delta$ neighbors, so each mixed class can carry about $n/\Delta$ anchors, and the $k$ classes about $kn/\Delta$ in total. About $n/k$ vertices of $A$ are left over by the $A$-only classes, since the class sizes $\floor{n/k}$ and $\ceil{n/k}$ need not divide $|A|$. Equating the total capacity $kn/\Delta$ with this demand $n/k$ yields $k=\Theta(\sqrt\Delta)$: above that threshold there are enough classes, below it there are not. Both bounds of this paper make the heuristic exact, in opposite directions --- and the factor $\sqrt{\ln\Delta}$ separating them is exactly the extent to which anchors can be chosen to share neighbors.

The upper bounds rest on a single combinatorial device, an \emph{anchored partition lemma} (Lemma~\ref{lem:anchor}). Given a partition of $V(G)$ into $\ell$ independent sets with largest part $B$, and an arbitrary prescribed profile of class sizes $s_1,\dots,s_k$, the lemma provides a pair of counting conditions under which $V(G)$ splits into independent sets of exactly these sizes, with a prescribed set of classes contained in a prescribed part, a prescribed set contained in $B$, and each remaining \emph{mixed} class anchored at a prescribed number of vertices of a prescribed part. The proof is a greedy argument with no backtracking. The improvement from $\Delta/2$ to $\sqrt\Delta$ comes from optimizing the number of colors against the two conditions: reserving $\theta k$ mixed classes leaves a reservoir of about $(\tfrac1\ell-\theta)n$ vertices of $B$; each mixed class can then absorb about $(\tfrac1\ell-\theta)n/\Delta$ leftover vertices, and about $(\ell-1)n/k$ require absorption. Balancing the two and optimizing $\theta$ gives $k=\Theta(\ell^{3/2}\sqrt\Delta)$ whenever $n=\Omega(\ell\Delta)$. Choosing the anchors adaptively rather than prescribing them replaces the estimate $|N(R)|\le|R|\Delta$ by one that is smaller by a factor $\Theta(\ln\Delta)$, which is where the saving of $\sqrt{\ln\Delta}$ in Theorem~\ref{thm:multi} comes from.

The lower bound takes a part $B$ of size about $\zeta\Delta$ and joins each vertex of a part $A$ to a uniformly random $\Delta$-subset of $B$. A concentration argument for negatively associated random variables shows that, with high probability, every set of $s_0=\Theta(\zeta\ln\Delta)$ vertices of $A$ dominates all but $O(\sqrt\Delta)$ vertices of $B$. Consequently no color class contains $s_0$ vertices of $A$ --- otherwise its $B$-side would be too small for the class to reach the equitable size $\floor{n/k}$ --- so at least $|A|/s_0$ classes meet $A$. The bound improves with $|A|$, which is capped at $\floor{n/k}$ so that the independent set $A$ cannot fill a class unaided. Taking $|A|$ at that ceiling gives $k\ge|A|/s_0=\Theta\bigl(n/(k\zeta\ln\Delta)\bigr)$, hence $k^2=\Theta(\Delta/\ln\Delta)$, in which $\zeta$ cancels.

\subsection*{Notation}
Throughout, $n:=|V(G)|$, $m:=|E(G)|$, and $\Delta:=\Delta(G)$. When $G$ is $\ell$-colorable we fix a proper $\ell$-coloring $V_1,\dots,V_\ell$ with $|V_\ell|$ largest, write $B:=V_\ell$ and $b:=|B|$, so that $b\ge n/\ell$, and set $a_j:=|V_j|$ for $j<\ell$ and $A:=V(G)\setminus B$; for a bipartite graph this is a bipartition $\{A,B\}$ with $a:=|A|\le b$. Given the number of colors $k\ge 1$, we write $n=kq+r$ with $q:=\floor{n/k}$ and $0\le r<k$; an equitable $k$-coloring thus has exactly $r$ classes of size $q+1$ and $k-r$ classes of size $q$. For $v\in V(G)$, $N(v)$ is the neighborhood of $v$, and for $S\subseteq V(G)$, $N(S):=\bigcup_{v\in S}N(v)$. Logarithms are natural, and $x_+:=\max\{x,0\}$. We say a set $S$ \emph{dominates} $v$ if $v\in N(S)$.

\subsection*{Organization}
Section~\ref{sec:lemma} contains the two combinatorial tools, the anchored partition lemma and the normalized-form lemma; everything later rests on them. Section~\ref{sec:upper} proves Theorems~\ref{thm:upper} and~\ref{thm:smallzeta}, together with Proposition~\ref{prop:threshsharp}, from a single proposition encapsulating the requirements of the construction in terms of $k$ and the number of reserved classes. Section~\ref{sec:lower} proves Theorem~\ref{thm:lower}. Section~\ref{sec:multi} replaces the prescribed anchors by adaptively chosen ones and proves Theorems~\ref{thm:multi} and~\ref{thm:uniform}; it is independent of Section~\ref{sec:upper} apart from the two tools of Section~\ref{sec:lemma}. Section~\ref{sec:alg} proves Theorem~\ref{thm:alg}, Section~\ref{sec:concl} discusses the limits of the method and what it leaves open, and Appendix~\ref{app:exp} reports numerical experiments.

\section{The anchored partition lemma}\label{sec:lemma}

Although stated for an arbitrary number of parts, Lemma~\ref{lem:anchor} is primarily intended for the bipartite setting, with parts $A,B$ as in the notation above. The colorings we build have three kinds of class (Figure~\ref{fig:anchor}): some are filled entirely from $A$, some entirely from $B$, and the remaining \emph{mixed} classes absorb whatever vertices of $A$ the first kind cannot accommodate, each anchored at a small subset of $A$ and completed greedily within $B$. Lemma~\ref{lem:anchor} isolates exactly what such a construction requires, for an \emph{arbitrary} prescribed profile of class sizes. Prescribing the sizes at the outset keeps the argument short: no class ever has to be resized, so the greedy filling never backtracks; the same feature yields the linear-time implementation of Section~\ref{sec:alg}.

\begin{lemma}[Anchored partition lemma]\label{lem:anchor}
Let $G$ be a graph with $\Delta:=\Delta(G)\ge 1$, let $\ell\ge2$, and let $V(G)=V_1\cup\dots\cup V_\ell$ be a partition into independent sets; write $B:=V_\ell$ and $A:=V(G)\setminus B$. Let $s_1,\dots,s_k\ge 1$ be integers with $\sum_{i=1}^{k}s_i=|V(G)|$. Suppose that $[k]$ can be partitioned into three (possibly empty) sets $\IA,\IM,\IB$, that every $i\in\IA\cup\IM$ carries a label $j(i)\in[\ell-1]$, and that there are integers $\rho_i$ with $0\le\rho_i\le s_i$ for $i\in\IM$, such that
\begin{enumerate}
\item[\textup{(P1)}] $\displaystyle\sum_{i\in\IA:\,j(i)=j}s_i+\sum_{i\in\IM:\,j(i)=j}\rho_i \;=\; |V_j|$ \quad for every $j\in[\ell-1]$, and
\item[\textup{(P2)}] $\displaystyle\sum_{i\in\IM}s_i+(\Delta-1)\max_{i\in\IM}\rho_i \;\le\; |B|$ \quad(\,$\max\emptyset:=0$\,).
\end{enumerate}
Then $V(G)$ can be partitioned into independent sets $D_1,\dots,D_k$ with $|D_i|=s_i$ for all $i$, such that $D_i\subseteq V_{j(i)}$ for $i\in\IA$, $D_i\subseteq B$ for $i\in\IB$, and $D_i\cap A\subseteq V_{j(i)}$ with $|D_i\cap A|=\rho_i$ for $i\in\IM$. Moreover, the partition can be constructed greedily, with arbitrary choices at every step.
\end{lemma}

For $\ell=2$ this is the bipartite statement, with $A=V_1$, and \textup{(P1)} is the single equation $\sum_{i\in\IA}s_i+\sum_{i\in\IM}\rho_i=|A|$; that is the only case used before Section~\ref{sec:multi}. The allowance for $\ell>2$ accommodates the case in which $A$ is not independent: it is enough that each class, and each anchor set, lies inside a single part.

\begin{proof}
Since each $V_j$ is independent, any family of disjoint subsets of a single $V_j$ consists of independent sets.

\emph{Phase 1.} By (P1) we have $\sum_{i\in\IA:\,j(i)=j}s_i\le|V_j|$ for each $j<\ell$, so we may choose pairwise disjoint sets $D_i\subseteq V_{j(i)}$ with $|D_i|=s_i$ for $i\in\IA$, arbitrarily. Let $R^{(j)}:=V_j\setminus\bigcup_{i\in\IA:\,j(i)=j}D_i$ be the leftover vertices of $V_j$, and $R:=\bigcup_{j<\ell}R^{(j)}$; by (P1), $|R^{(j)}|=\sum_{i\in\IM:\,j(i)=j}\rho_i$.

\emph{Phase 2.} Enumerate $\IM=\{i_1,\dots,i_t\}$ in any order. We build $D_{i_1},\dots,D_{i_t}$ one at a time. When it is the turn of $i:=i_\sigma$, choose any set $R_i$ of $\rho_i$ as yet unused vertices of $R^{(j(i))}$ (possible, as the $\rho_{i'}$ with $j(i')=j(i)$ sum to $|R^{(j(i))}|$), and let $D_i:=R_i\cup T_i$ where $T_i$ consists of any $s_i-\rho_i$ vertices of $B$ that are neither contained in an earlier mixed class nor adjacent to $R_i$. Such vertices exist: the earlier mixed classes contain $\sum_{u<\sigma}(s_{i_u}-\rho_{i_u})$ vertices of $B$, while $|N(R_i)|\le\rho_i\Delta$, so the number of admissible vertices is at least
\[
|B|-\sum_{u<\sigma}\bigl(s_{i_u}-\rho_{i_u}\bigr)-\rho_i\Delta
\;\ge\;
|B|-\sum_{u\le\sigma} s_{i_u}+\bigl(s_i-\rho_i\bigr)-\rho_i(\Delta-1)
\;\ge\; s_i-\rho_i,
\]
where the last inequality uses $\sum_{u\le\sigma}s_{i_u}\le\sum_{i'\in\IM}s_{i'}$ together with $\rho_i\le\max_{i'\in\IM}\rho_{i'}$ and (P2). Since $R_i\subseteq V_{j(i)}$ is independent, $T_i\subseteq B$ is independent, and no vertex of $T_i$ has a neighbor in $R_i$, the set $D_i$ is independent and $|D_i\cap A|=\rho_i$.

\emph{Phase 3.} After Phase~2 every vertex of $A$ is used. Summing (P1) over $j$ gives $\sum_{i\in\IA}s_i+\sum_{i\in\IM}\rho_i=|A|$, so the number of unused vertices of $B$ equals
\[
|B|-\sum_{i\in\IM}\bigl(s_i-\rho_i\bigr)
= n-|A|-\sum_{i\in\IM}s_i+\Bigl(|A|-\sum_{i\in\IA}s_i\Bigr)
= \sum_{i\in\IB}s_i ,
\]
using $\sum_{i\in[k]}s_i=n$. Partition these vertices arbitrarily into sets $D_i$, $i\in\IB$, with $|D_i|=s_i$. All requirements are met.
\end{proof}

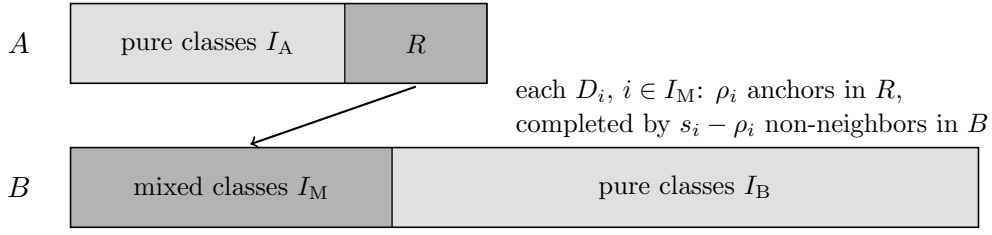
\begin{figure}[t]
\centering
\begin{tikzpicture}[xscale=1.25,yscale=0.95]
\draw[thick] (0,0) rectangle (4.4,1.1); \node at (-0.55,0.55) {$A$};
\draw[thick] (0,-2.0) rectangle (9.6,-0.9); \node at (-0.55,-1.45) {$B$};
\draw[fill=black!12] (0,0) rectangle (2.9,1.1);
\node at (1.45,0.55) {\small pure classes $\IA$};
\draw[fill=black!30] (2.9,0) rectangle (4.4,1.1);
\node at (3.65,0.55) {\small $R$};
\draw[fill=black!30] (0,-2.0) rectangle (3.4,-0.9);
\node at (1.7,-1.45) {\small mixed classes $\IM$};
\draw[fill=black!12] (3.4,-2.0) rectangle (9.6,-0.9);
\node at (6.5,-1.45) {\small pure classes $\IB$};
\draw[->,thick] (3.65,-0.05) -- (1.9,-0.85);
\node[anchor=west,align=left] at (4.6,-0.35) {\small each $D_i$, $i\in\IM$: $\rho_i$ anchors in $R$,\\ \small completed by $s_i-\rho_i$ non-neighbors in $B$};
\end{tikzpicture}
\caption{The three phases of Lemma~\ref{lem:anchor}, drawn for $\ell=2$. Condition (P1) says the shaded regions of $A$ tile $A$ exactly; condition (P2) reserves enough of $B$ for the mixed classes even after their anchors forbid up to $\rho_i\Delta$ vertices each.}
\label{fig:anchor}
\end{figure}

Two further tools convert Lemma~\ref{lem:anchor} into equitable colorings. The first distributes the leftover vertices of $A$ evenly among the mixed classes.

\begin{observation}\label{obs:distribute}
Let $t\ge 1$ and $0\le M\le tH$ for integers $M,H\ge0$. Then there are integers $\rho_1,\dots,\rho_t$ with $0\le\rho_i\le H$, $\sum_i\rho_i=M$ and $\max_i\rho_i\le\ceil{M/t}$: take $\rho_i=\ceil{M/t}$ for $i\le M-t\floor{M/t}$ and $\rho_i=\floor{M/t}$ otherwise. Since $M\le tH$, we have $\ceil{M/t}\le H$.
\end{observation}

The second is a normalized decomposition of the parts adapted to a $(q,q{+}1)$-profile. There are $k$ classes to distribute, $r$ of size $q+1$ and $k-r$ of size $q$. Each part $V_j$ other than $B$ should take as many of them as it can fill exactly, leaving a residue $M_j$ small enough for the mixed classes to absorb; and the parts must not between them exhaust the supply, since $t$ classes are reserved for mixing and the rest go to $B$. The lemma says that a single greedy pass, giving each part the largest class its remaining vertices can hold, achieves all of this: every residue is at most the size of one class, and at least $t$ classes survive. Hypothesis \eqref{eq:Hcond} guarantees the supply cannot run out, and it says exactly that $B$ is large enough to absorb the $t$ reserved classes. Only the residue bound $M_j\le q$ requires care; the remaining conclusions are immediate from the construction.

\begin{lemma}[Multipart normalized form]\label{lem:multinormal}
Let $q,k\ge1$, $0\le r<k$ and $n=kq+r$. Let $a_1,\dots,a_{\ell-1}\ge0$ and $b\ge0$ satisfy $\sum_{j}a_j+b=n$, and let $1\le t\le k$ be such that
\begin{equation}\label{eq:Hcond}
b\;\ge\;r+tq .
\end{equation}
Then there are integers $u_j,v_j,M_j\ge0$ $(1\le j\le\ell-1)$ with
\[
a_j=u_j(q+1)+v_jq+M_j,\qquad 0\le M_j\le q,
\]
\[
\sum_j u_j\le r,\qquad \sum_j v_j\le k-r,\qquad \sum_j\bigl(u_j+v_j\bigr)\le k-t .
\]
In particular $\sum_j M_j\le(\ell-1)q$, and setting $y:=k-t-\sum_j(u_j+v_j)$ one has $y\ge0$ and $0\le r-\sum_j u_j\le t+y$.
\end{lemma}

\begin{proof}
Maintain budgets $r'$ and $s'$, initially $r'=r$ and $s'=k-r$, counting the still unallocated classes of size $q+1$ and of size $q$. Process the parts in any order. For part $j$, let $x$ denote the number of its vertices not yet allocated, initially $a_j$, and repeat: if $x\ge q+1$ and $r'>0$, allocate a class of size $q+1$, decreasing $r'$ and $x$ accordingly; otherwise, if $x\ge q$ and $s'>0$, allocate a class of size $q$; otherwise stop. Let $u_j$ and $v_j$ be the numbers of classes of each size allocated to part $j$, and $M_j$ the final value of $x$. The identity $a_j=u_j(q+1)+v_jq+M_j$ and the two budget inequalities hold by construction.

We claim $M_j\le q$. Since the procedure stops at part $j$ only when no further class can be allocated, it suffices to rule out a stop with $x\ge q+1$, and we show more: a stop with $x\ge q$ forces $x=q$ exactly. So suppose the procedure stops at part $j$ with $x\ge q$. A class of size $q$ was not allocated although $x\ge q$, so the reason must be $s'=0$. If also $r'=0$ then all $k$ classes have been allocated, so the total number of allocated vertices is $\sum_i\bigl(u_i(q+1)+v_iq\bigr)=r(q+1)+(k-r)q=n$; but that total is at most $\sum_i a_i=n-b$, forcing $b\le0$, whereas \eqref{eq:Hcond} gives $b\ge r+tq\ge q\ge1$. Hence $r'>0$, so the only possible reason a class of size $q+1$ was not allocated is $x<q+1$. Thus $x=q$, and $M_j=q$.

For the class count, every allocated class has size at least $q$, so
\[
q\sum_j\bigl(u_j+v_j\bigr)\;\le\;\sum_j\bigl(u_j(q+1)+v_jq\bigr)\;\le\;\sum_j a_j\;=\;n-b\;\le\;n-r-tq\;=\;(k-t)q
\]
by \eqref{eq:Hcond}, whence $\sum_j(u_j+v_j)\le k-t$ and $y\ge0$. Finally $r-\sum_ju_j\ge0$ by the budget on $r'$, and $r-\sum_j u_j\le k-\sum_j u_j-\sum_j v_j=t+y$ because $\sum_j v_j\le k-r$.
\end{proof}

\begin{remark}\label{rem:special}
The bound $M_j\le q$ determines the residues that the mixed classes must absorb, and with them the hypothesis on the order in Theorem~\ref{thm:upper} (cf.\ Remark~\ref{rem:special2}). It is attained (Appendix~\ref{app:exp}), so the lemma cannot be improved as stated. Specializing Lemma~\ref{lem:anchor} and Lemma~\ref{lem:multinormal} to $\ell=2$ recovers the covering step used in the linear-bound literature, and in particular the covering lemma of~\cite[Lemma~2.2]{Nikabadi26}; arguments of that shape end with a rebalancing step --- moving vertices between classes to correct the sizes --- which is unnecessary here because the sizes are prescribed at the outset.
\end{remark}

\section{The prescribed-anchor bound}\label{sec:upper}

Throughout this section $G$ is $\ell$-colorable, with $V_1,\dots,V_\ell$, $B$, $b$, $a_j$, $q$ and $r$ as in the Notation. Given the number of colors $k$ and a number $t$ of classes to reserve for mixing, put
\begin{equation}\label{eq:betaLH}
\beta\;:=\;b-t(q+1),\qquad L\;:=\;\floor*{\frac{\beta}{\Delta-1}},\qquad H\;:=\;\min\{q,L\} .
\end{equation}
The following proposition encapsulates the requirements of the construction for a single $k$ and a single $t$. It makes no assumption relating $k$ to $\sqrt\Delta$, and the freedom to choose $t$ independently of $k$ yields the threshold statement.

\begin{proposition}[The core construction]\label{prop:core}
Let $\ell\ge2$, let $G$ be as above with $\Delta\ge2$, and let $k,t\ge1$ be integers. Suppose that
\begin{enumerate}
\item[\textup{(A0)}] $q\;\ge\;1$;
\item[\textup{(A1)}] $b\;\ge\;r+tq$;
\item[\textup{(A2)}] $L\;\ge\;1$;
\item[\textup{(A3)}] $(\ell-1)\ceil*{q/H}\;\le\;t$.
\end{enumerate}
Then $G$ has an equitable $k$-coloring, produced by the greedy construction of Lemma~\textup{\ref{lem:anchor}} with all anchor sets prescribed in advance.
\end{proposition}

\begin{proof}
By (A0) the profile is well defined. Also $t\le k$, as Lemma~\ref{lem:multinormal} requires, since (A1) gives $tq\le b-r\le n-r=kq$ and $q\ge1$ by (A0). Condition (A1) is hypothesis \eqref{eq:Hcond}, so Lemma~\ref{lem:multinormal} supplies integers $u_j,v_j,M_j$ with $a_j=u_j(q+1)+v_jq+M_j$ and $M_j\le q$, with $\sum_j(u_j+v_j)\le k-t$, and, setting $y:=k-t-\sum_j(u_j+v_j)\ge0$, with $0\le r-\sum_ju_j\le t+y$. By (A2), $H\ge1$, so $t_j:=\ceil{M_j/H}$ is well defined, and $\sum_{j<\ell}t_j\le(\ell-1)\ceil{q/H}\le t$ by (A3) and $M_j\le q$.

Apply Lemma~\ref{lem:anchor} with the given $\ell$-partition and the profile consisting of $r$ classes of size $q+1$ and $k-r$ of size $q$. Let $\IA$ consist of $\sum_j(u_j+v_j)$ classes, those labeled $j$ comprising $u_j$ of size $q+1$ and $v_j$ of size $q$; let $\IM$ consist of $\sum_jt_j$ further classes, $t_j$ of them labeled $j$; and let $\IB$ be the rest. Then $|\IM|\le t$ and $|\IM\cup\IB|=t+y$, so the remaining $r-\sum_ju_j$ classes of size $q+1$ can be placed among them. Since $M_j\le t_jH$, Observation~\ref{obs:distribute} yields integers $\rho_i\le H\le q\le s_i$ for the $t_j$ classes labeled $j$, with $\sum_{j(i)=j}\rho_i=M_j$. Condition (P1) holds for every $j$, since $u_j(q+1)+v_jq+M_j=a_j$, and (P2) holds because
\[
\sum_{i\in\IM}s_i+(\Delta-1)\max_{i\in\IM}\rho_i\;\le\;t(q+1)+(\Delta-1)L\;\le\;t(q+1)+\beta\;=\;b .
\]
Lemma~\ref{lem:anchor} partitions $V(G)$ into $k$ independent sets, exactly $r$ of size $q+1$ and $k-r$ of size $q$: an equitable $k$-coloring.
\end{proof}

Conditions (A0)--(A3) are monotone in $k$ once $t$ is held fixed:

\begin{lemma}\label{lem:mono}
Let $k_1\le k$ be integers with $q(k)\ge1$, let $t\ge1$, and suppose $r(k)+tq(k)\le b$. If \textup{(A2)} and \textup{(A3)} hold at $k_1$, they hold at $k$.
\end{lemma}

\begin{proof}
Write $q_1:=q(k_1)$ and $q:=q(k)$, and let $\beta_1,L_1,H_1$ and $\beta,L,H$ be the corresponding quantities \eqref{eq:betaLH}. Since $k\ge k_1$ we have $q=\floor{n/k}\le\floor{n/k_1}=q_1$, hence $\beta=b-t(q+1)\ge\beta_1$ and $L\ge L_1\ge1$, which is (A2). For (A3), if $L\ge q$ then $H=q$ and $\ceil{q/H}=1$, so $(\ell-1)\ceil{q/H}=\ell-1\le(\ell-1)\ceil{q_1/H_1}\le t$. Otherwise $H=L$ and
\[
\ceil*{\frac qH}\;=\;\ceil*{\frac qL}\;\le\;\ceil*{\frac{q_1}{L_1}}\;=\;\ceil*{\frac{q_1}{H_1}},
\]
the last equality because $L<q$ forces $L_1\le L<q\le q_1$, so $H_1=L_1$ as well. Hence $(\ell-1)\ceil{q/H}\le t$.
\end{proof}

We can now prove Theorem~\ref{thm:upper}, in its threshold form.

\begin{proof}[Proof of Theorem~\textup{\ref{thm:upper}}]
Fix $\ell\ge2$, $\zeta>2\ell$ and $\gamma>2\ell\sqrt{\zeta(\ell-1)/(\zeta-2\ell)}$; the latter is equivalent to
\begin{equation}\label{eq:leading}
\zeta\bigl(\gamma^2-4\ell^2(\ell-1)\bigr)\;>\;2\ell\gamma^2 .
\end{equation}
All inequalities below compare polynomials in $k_1$, so comparing leading coefficients determines an explicit $\Delta_0=\Delta_0(\ell,\zeta,\gamma)$: fix it so large that for every $\Delta\ge\Delta_0$, writing $k_1:=\ceil{\gamma\sqrt\Delta}$,
\begin{enumerate}
\item[(C0)] $\Delta\ge2$,\quad $k_1\ge2\ell^2$,\quad $k_1\le n/(2\ell)$,\quad and\quad $n>k_1+2\ell\Delta$;
\item[(C1)] $\gamma^2(k_1-2\ell^2)\;>\;4\ell^2(\ell-1)k_1$;
\item[(C2)] $\zeta(k_1-1)^2\bigl(\gamma^2(k_1-2\ell^2)-4\ell^2(\ell-1)k_1\bigr)\;\ge\;\gamma^2\bigl(2\ell k_1^3+\gamma^2k_1^2\bigr)$.
\end{enumerate}
Here $k_1\to\infty$ with $\Delta$ while $n\ge\zeta\Delta\ge\zeta(k_1-1)^2/\gamma^2$ grows quadratically in $k_1$, which gives (C0); (C1) is clear; and in (C2) both sides are cubics in $k_1$ with leading coefficients $\zeta(\gamma^2-4\ell^2(\ell-1))$ and $2\ell\gamma^2$, so \eqref{eq:leading} makes (C2) valid for all large $k_1$.

Let $G$ be $\ell$-colorable with $\Delta\ge\Delta_0$ and $n\ge\zeta\Delta$, and set $t:=\floor{k_1/(2\ell)}\ge1$, so that
\begin{equation}\label{eq:kDelta}
\frac{(k_1-1)^2}{\gamma^2}\;\le\;\Delta\;\le\;\frac{k_1^2}{\gamma^2},
\qquad\text{hence}\qquad
n\;\ge\;\zeta\Delta\;\ge\;\frac{\zeta(k_1-1)^2}{\gamma^2} .
\end{equation}

\emph{Step 1: the hypotheses hold at $k_1$.} Condition (A0) holds, that is, $q_1\ge1$, because $k_1\le n/(2\ell)\le n$ by (C0). Since $b\ge n/\ell$, $tq_1\le\frac{k_1}{2\ell}\cdot\frac n{k_1}=\frac n{2\ell}$ and $r<k_1\le n/(2\ell)$, condition (A1) holds. For (A2) we record the estimate on $L_1$ once and in full. Since $t\le k_1/(2\ell)$ and $q_1\le n/k_1$ we have $t(q_1+1)\le\frac{n+k_1}{2\ell}$, so $\beta_1\ge\frac n\ell-\frac{n+k_1}{2\ell}=\frac{n-k_1}{2\ell}$, and therefore
\begin{equation}\label{eq:Lbound}
L_1\;=\;\floor*{\frac{\beta_1}{\Delta-1}}\;\ge\;\floor*{\frac{\beta_1}{\Delta}}\;\ge\;\frac{\beta_1}{\Delta}-1\;\ge\;\frac{n-k_1}{2\ell\Delta}-1\;=\;\frac{n-k_1-2\ell\Delta}{2\ell\Delta}\;>\;0 ,
\end{equation}
the last inequality by the clause $n>k_1+2\ell\Delta$ of (C0); in particular $L_1\ge1$. Only the third step discards a floor: subtracting $1$ removes $2\ell\Delta$ from the numerator, the term provided for in (C0); the final expression in \eqref{eq:Lbound} is the one used below. For (A3), if $L_1\ge q_1$ then $\ceil{q_1/H_1}=1$ and $(\ell-1)\le t$ by (C0). Otherwise $H_1=L_1$, and since $\ceil{q_1/L_1}\le q_1/L_1+1$ it suffices that $(\ell-1)q_1/L_1\le t-(\ell-1)$. Here $q_1\le n/k_1$, and $t\ge(k_1-2\ell)/(2\ell)$ gives $t-(\ell-1)\ge(k_1-2\ell^2)/(2\ell)$, while $L_1\ge\frac{n-k_1-2\ell\Delta}{2\ell\Delta}>0$ by \eqref{eq:Lbound}. It therefore suffices that
\[
4\ell^2(\ell-1)\,n\Delta\;\le\;k_1(k_1-2\ell^2)\bigl(n-2\ell\Delta-k_1\bigr) .
\]
By \eqref{eq:kDelta} we may replace $\Delta$ by $k_1^2/\gamma^2$ on the left and in the term $-2\ell\Delta$ on the right; multiplying by $\gamma^2/k_1$ and rearranging, it suffices that
\[
n\bigl(\gamma^2(k_1-2\ell^2)-4\ell^2(\ell-1)k_1\bigr)\;\ge\;(k_1-2\ell^2)\bigl(\gamma^2k_1+2\ell k_1^2\bigr).
\]
The coefficient of $n$ is positive by (C1), so by \eqref{eq:kDelta} the left side is at least $\zeta(k_1-1)^2\bigl(\gamma^2(k_1-2\ell^2)-4\ell^2(\ell-1)k_1\bigr)/\gamma^2$, while the right side is at most $\gamma^2k_1^2+2\ell k_1^3$. This is (C2) divided by $\gamma^2$. Proposition~\ref{prop:core} therefore gives an equitable $k_1$-coloring.

\emph{Step 2: the hypotheses persist for $k_1\le k\le\Delta$.} Keep $t$ at the value $\floor{k_1/(2\ell)}$ just fixed. Then $q(k)\ge\floor{n/\Delta}\ge\floor{\zeta}\ge1$, which is (A0), and, since $r<k\le\Delta\le n/\zeta\le n/(2\ell)$ and $tq(k)\le tq_1\le n/(2\ell)$, we get $r+tq(k)\le n/\ell\le b$, which is (A1). Conditions (A2) and (A3) follow from Lemma~\ref{lem:mono} and Step~1. So Proposition~\ref{prop:core} gives an equitable $k$-coloring for every such $k$.

\emph{Step 3: large $k$.} For $k\ge\Delta+1$ the Hajnal--Szemer\'edi theorem~\cite{HajnalSzemeredi70}, applied with maximum-degree parameter $k-1\ge\Delta$, gives an equitable $k$-coloring. Hence $G$ is equitably $k$-colorable for every $k\ge k_1$, that is, $\chies(G)\le k_1$. \qedhere
\end{proof}

\begin{remark}\label{rem:frozen}
Freezing $t$ at $\floor{k_1/(2\ell)}$ for the whole range makes Step~2 immediate. Reserving $t$ classes costs $t(q+1)$ vertices of $B$, and $q$ decreases as $k$ grows, so a reservation adequate at $k_1$ is more than adequate at every larger $k$; meanwhile the demand $(\ell-1)\ceil{q/H}$ only falls. Letting $t$ grow with $k$, as would be natural, destroys this monotonicity and forces a separate analysis at the upper endpoint, which costs the earlier treatments of this range an additive constant in the hypothesis on the order.
\end{remark}

\subsection*{Bipartite graphs with little room}
We now prove Theorem~\ref{thm:smallzeta}. The missing ingredient is that when the two parts are of comparable size no anchoring is needed at all --- the vertex set splits into classes lying wholly inside one part, and only an arithmetic condition has to be met.

\begin{lemma}[Arithmetic split]\label{lem:arith}
Let the vertex set of a graph be partitioned into two independent sets $A,B$ with $a:=|A|\le|B|=:b$, and let $q\ge1$ be an integer with $q^2\le a$ and $4q^2+2q\le n$. Then the graph has an equitable $k$-coloring with $k:=\floor{n/q}-1$.
\end{lemma}

\begin{proof}
Since $kq\le n-q$ we have $n-kq\ge q\ge1$, so $\floor{n/k}\ge q$; and $(q+1)k\ge(q+1)(\tfrac nq-2)=n+\tfrac nq-2q-2>n$ because $n>2q(q+1)$, so $\floor{n/k}=q$. Writing $n=kq+r$ gives $r=q+(n\bmod q)$, hence
\begin{equation}\label{eq:rrange}
q\;\le\;r\;\le\;2q-1 .
\end{equation}
Put $x:=\floor{a/q}$ and $u:=a\bmod q$, so that $a=xq+u$ and $0\le u\le q-1$. From $q^2\le a$ we get $x\ge q>u$, and \eqref{eq:rrange} gives $u<q\le r$. Also $x\le a/q\le n/(2q)\le k$ and, using $b\ge n/2$,
\[
k-x\;\ge\;\Bigl(\frac nq-2\Bigr)-\frac aq\;=\;\frac bq-2\;\ge\;\frac{n}{2q}-2\;\ge\;2q-1\;\ge\;r-u ,
\]
the third inequality by $4q^2+2q\le n$. Give $A$ exactly $x$ classes, $u$ of size $q+1$ and $x-u$ of size $q$, and give $B$ the remaining $k-x$ classes, $r-u$ of them of size $q+1$; the counts match because $b=(k-x)q+(r-u)$. Every class lies inside $A$ or inside $B$, hence is independent.
\end{proof}

We restate Theorem~\ref{thm:smallzeta} in the form proved: for every $\zeta>1$ and every $\gamma>C(\zeta):=\max\bigl\{4\sqrt{\zeta/(\zeta-1)},\,2\sqrt{10}\bigr\}$ there is $\Delta_0=\Delta_0(\zeta,\gamma)$ such that every bipartite graph $G$ with $\Delta\ge\Delta_0$ and $n\ge\zeta\Delta$ satisfies $\chie(G)\le\ceil{\gamma\sqrt\Delta}$.

\begin{proof}[Proof of Theorem~\textup{\ref{thm:smallzeta}}]
Let $A,B$ be the parts with $a:=|A|\le b:=|B|$. A vertex of maximum degree lies in one part and all its neighbors in the other, so
\begin{equation}\label{eq:bgeD}
b\;\ge\;\Delta .
\end{equation}
If $n\ge8\Delta$ then Theorem~\ref{thm:upper}, applied with $\ell=2$, $\zeta=8$ and any $\gamma'\in(4\sqrt2,\gamma)$, already gives the conclusion, since $4\sqrt2=5.657\ldots<2\sqrt{10}\le\gamma$. So assume $\zeta\Delta\le n<8\Delta$, and set $\alpha^{*}:=(\zeta-1)/(2\zeta)$.

\emph{Case 1: $a\ge\alpha^{*}n$.} Take $q:=\floor{\sqrt{\min\{a,n/5\}}}$. Then $q^2\le a$, and $4q^2+2q\le\tfrac45n+2\sqrt n\le n$ for large $n$, so Lemma~\ref{lem:arith} gives an equitable $k$-coloring with
\[
k\;=\;\floor*{\frac nq}-1\;\le\;\bigl(1+o(1)\bigr)\max\Bigl\{\frac{n}{\sqrt a},\ \sqrt{5n}\Bigr\}
\;\le\;\bigl(1+o(1)\bigr)\max\Bigl\{\sqrt{\frac{n}{\alpha^{*}}},\ \sqrt{5n}\Bigr\} .
\]
Since $n<8\Delta$, this is at most $\bigl(1+o(1)\bigr)\max\bigl\{4\sqrt{\zeta/(\zeta-1)},\,2\sqrt{10}\bigr\}\sqrt\Delta\le\gamma\sqrt\Delta$ for large $\Delta$.

\emph{Case 2: $a<\alpha^{*}n$.} Then, using $\Delta\le n/\zeta$ and \eqref{eq:bgeD},
\begin{equation}\label{eq:bDelta}
b-\Delta\;=\;n-a-\Delta\;>\;n\bigl(1-\alpha^{*}\bigr)-\frac n\zeta\;=\;\frac{(\zeta-1)n}{2\zeta}\;>\;0 .
\end{equation}
Set $k:=\ceil{\gamma\sqrt\Delta}$ and $q:=\floor{n/k}\ge1$, and apply Proposition~\ref{prop:core} with $\ell=2$ and $t:=q$; here $t\le k$, since $q\le n/k<(8/\gamma)\sqrt\Delta\le\gamma\sqrt\Delta\le k$ because $\gamma>2\sqrt{10}$. Condition (A3) reads $\ceil{q/H}\le q$, which holds as soon as $H\ge1$; condition (A2) says $L\ge1$, that is, $\beta=b-q(q+1)\ge\Delta-1$, or equivalently
\[
q(q+1)\;\le\;b-\Delta+1 ;
\]
and (A1) then follows, since $q^2\le b-\Delta$ and $r<k\le\Delta$ give $r+tq=r+q^2\le\Delta+b-\Delta=b$. Now $q\le n/k\le n/(\gamma\sqrt\Delta)$, so $q(q+1)\le(1+o(1))n^2/(\gamma^2\Delta)$, and by \eqref{eq:bDelta} it suffices that
\[
\frac{n^2}{\gamma^2\Delta}\;\le\;\bigl(1-o(1)\bigr)\frac{(\zeta-1)n}{2\zeta},
\qquad\text{that is}\qquad
\gamma^2\;\ge\;\bigl(1+o(1)\bigr)\frac{2\zeta}{\zeta-1}\cdot\frac{n}{\Delta} .
\]
As $n<8\Delta$, this is implied by $\gamma\ge(1+o(1))\,4\sqrt{\zeta/(\zeta-1)}$, which holds by hypothesis. Proposition~\ref{prop:core} now gives an equitable $k$-coloring. \qedhere
\end{proof}

\begin{proposition}\label{prop:threshsharp}
For every odd $a$ the complete bipartite graph $K_{a,a}$ satisfies $\Delta=a$, $n=2\Delta$, $\chie(K_{a,a})=2$ and $\chies(K_{a,a})=\Delta+1$.
\end{proposition}

\begin{proof}
Every color class lies inside one part, and the two parts form an equitable $2$-coloring, so $\chie=2$. At $k=a$ we have $n=2a=2k$, so $q=2$ and $r=0$: every class has exactly two vertices, and a part of odd size $a$ cannot be partitioned into pairs. Hence $K_{a,a}$ is not equitably $a$-colorable and $\chies>a=\Delta$, while $\chies\le\Delta+1$ by the Hajnal--Szemer\'edi theorem.
\end{proof}

\begin{remark}\label{rem:special2}
The reservation fraction $1/(2\ell)$ is optimal for this argument: writing $t=\floor{\theta k}$, the reservoir is $\beta\approx(\tfrac1\ell-\theta)n$, so a mixed class absorbs about $(\tfrac1\ell-\theta)n/\Delta$ vertices of the residues, while the residues total about $(\ell-1)n/k$; matching supply against demand requires $\gamma^2\theta(\tfrac1\ell-\theta)\ge\ell-1$ to leading order. The left side is maximized at $\theta=1/(2\ell)$, where it equals $\gamma^2/(4\ell^2)$ and returns the condition $\gamma>2\ell\sqrt{\ell-1}$, the limiting case $\zeta\to\infty$ of Theorem~\ref{thm:upper}.
\end{remark}

\section{The lower bound}\label{sec:lower}

We prove Theorem~\ref{thm:lower}. The construction is probabilistic; we first recall the concentration bound we use. A family of random variables $Y_1,\dots,Y_N$ is \emph{negatively associated (NA)} if for every pair of disjoint sets $I,J\subseteq[N]$ and every pair of coordinatewise nondecreasing functions $f,g$,
\[
\E\bigl[f(Y_i,\,i\in I)\,g(Y_j,\,j\in J)\bigr]\;\le\;\E\bigl[f(Y_i,\,i\in I)\bigr]\,\E\bigl[g(Y_j,\,j\in J)\bigr].
\]
We use three standard facts from Joag-Dev and Proschan~\cite{JoagDevProschan83}: (a) a uniformly random permutation of a fixed real vector yields an NA family (permutation distributions are NA); (b) the union of independent NA families is NA; (c) coordinatewise nondecreasing functions of pairwise disjoint subfamilies of an NA family form an NA family. Moreover, the Chernoff--Hoeffding upper-tail bound applies verbatim to sums of NA indicator variables: the standard proof bounds $\E e^{sX}$ by $\prod_v\E e^{sY_v}$ for $s>0$, and this factorization is exactly what negative association supplies, since $y\mapsto e^{sy}$ is nondecreasing (Dubhashi and Ranjan~\cite{DubhashiRanjan98}). Thus if $X=Y_1+\dots+Y_N$ with $(Y_v)$ NA indicators and $\mu\ge\E X$, then for all $x>\mu$,
\begin{equation}\label{eq:chernoff}
\Prb{X\ge x}\;\le\;\Bigl(\frac{e\mu}{x}\Bigr)^{x}.
\end{equation}

\begin{proof}[Proof of Theorem~\textup{\ref{thm:lower}}]
Fix $\ell\ge2$ and $\zeta>1$, and let $\Delta$ be large in terms of $\ell$ and $\zeta$. Set
\begin{equation}\label{eq:lbparams}
s_0:=\ceil*{2\zeta\ln\Delta},\quad K:=\floor*{\tfrac12\sqrt{\tfrac{(\ell-1)\Delta}{\ln\Delta}}},\quad
a:=\ceil*{\frac{K\,s_0}{\ell-1}},\quad b:=\ceil*{\zeta\Delta},\quad n:=(\ell-1)a+b,
\end{equation}
and $x_0:=\ceil*{\sqrt\Delta}$, $\Delta':=\Delta-(\ell-2)a$. Let $V_1,\dots,V_{\ell-1}$ be disjoint sets of size $a$ and let $B$ be a disjoint set of size $b$. Build $G$ in two steps (Figure~\ref{fig:lower}): join every pair of vertices lying in \emph{different} sets $V_i,V_j$ $(i\ne j<\ell)$, and let each $u\in\bigcup_{j<\ell}V_j$ choose its neighborhood in $B$ independently and uniformly among the $\Delta'$-subsets of $B$.

From \eqref{eq:lbparams}, $(\ell-1)a\le Ks_0+\ell\le\zeta\sqrt{(\ell-1)\Delta\ln\Delta}+K+\ell$, so $(\ell-1)a\le\Delta/10$ for large $\Delta$; hence $\Delta'\ge\tfrac9{10}\Delta$ and $\Delta'\le\Delta\le b$, and the construction is well defined. Every vertex of $V_j$ has degree exactly $(\ell-2)a+\Delta'=\Delta$ and every vertex of $B$ has degree at most $(\ell-1)a\le\Delta$, so $\Delta(G)=\Delta$ deterministically; $n\ge b\ge\zeta\Delta$; and $V_1,\dots,V_{\ell-1},B$ are independent sets, so $\chi(G)\le\ell$.

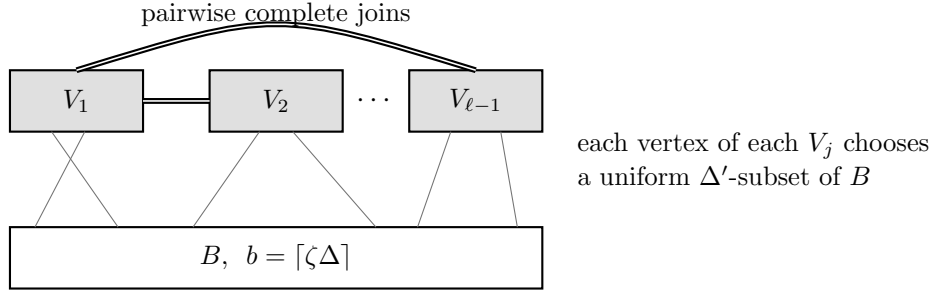
\begin{figure}[t]
\centering
\begin{tikzpicture}[xscale=1.1,yscale=0.9]
\draw[thick,fill=black!12] (0,0) rectangle (1.6,0.9);   \node at (0.8,0.45) {\small $V_1$};
\draw[thick,fill=black!12] (2.4,0) rectangle (4.0,0.9); \node at (3.2,0.45) {\small $V_2$};
\draw[thick,fill=black!12] (4.8,0) rectangle (6.4,0.9); \node at (5.6,0.45) {\small $V_{\ell-1}$};
\node at (4.4,0.45) {$\cdots$};
\draw[double,thick] (1.6,0.45) -- (2.4,0.45);
\draw[double,thick] (0.8,0.9) .. controls (3.2,1.8) .. (5.6,0.9);
\node at (3.2,1.72) {\small pairwise complete joins};
\draw[thick] (0,-2.3) rectangle (6.4,-1.4);
\node at (3.2,-1.85) {\small $B$, \ $b=\ceil{\zeta\Delta}$};
\draw[black!55] (0.5,0) -- (1.3,-1.4);
\draw[black!55] (0.9,0) -- (0.3,-1.4);
\draw[black!55] (3.0,0) -- (2.2,-1.4);
\draw[black!55] (3.4,0) -- (4.4,-1.4);
\draw[black!55] (5.3,0) -- (4.9,-1.4);
\draw[black!55] (5.9,0) -- (6.1,-1.4);
\node[align=left,anchor=west] at (6.7,-0.4) {\small each vertex of each $V_j$ chooses\\ \small a uniform $\Delta'$-subset of $B$};
\end{tikzpicture}
\caption{The construction of Theorem~\ref{thm:lower}. The parts $V_1,\dots,V_{\ell-1}$, each of size $a$, are pairwise completely joined, so no color class meets two of them; each of their vertices chooses its neighborhood in $B$ uniformly and independently among the $\Delta'$-subsets. Every vertex of $\bigcup_{j<\ell}V_j$ has degree exactly $\Delta$, and with high probability every set of $s_0$ vertices of one part dominates all but at most $x_0$ vertices of $B$.}
\label{fig:lower}
\end{figure}

The argument hinges on two features. The sets $V_j$ are pairwise completely joined, so \emph{every color class meets at most one of them}. And each $V_j$ is as large as possible: the obstruction below counts classes as $|V_j|$ divided by the number of vertices of $V_j$ that one class can hold, and the only ceiling on $a$ is that $V_j$ must not fill a class unaided, since $V_j$ is independent and a class contained in it evades the obstruction. The choice \eqref{eq:lbparams} sits just below that ceiling, by \eqref{eq:gap} below.

Call $S\subseteq V_j$ \emph{thorough} if at most $x_0$ vertices of $B$ have no neighbor in $S$. We claim that with high probability
\begin{equation}\label{eq:eventD}
\text{for every }j<\ell,\text{ every }S\in\tbinom{V_j}{s_0}\text{ is thorough.}
\end{equation}
Fix such an $S$ and for $v\in B$ let $Y_v$ be the indicator that $v\notin N(S)$. Since $b=\ceil{\zeta\Delta}\le\tfrac{11}{10}\zeta\Delta$ for $\Delta$ large and $\Delta'\ge\tfrac9{10}\Delta$, each $u\in S$ misses a fixed $v$ with probability $1-\Delta'/b\le e^{-\Delta'/b}\le e^{-9/(11\zeta)}$, independently over $u\in S$; hence, using $s_0\ge2\zeta\ln\Delta$,
\[
\E\Bigl[\sum_{v\in B}Y_v\Bigr]\;\le\;b\,e^{-9s_0/(11\zeta)}\;\le\;\tfrac{11}{10}\zeta\Delta\cdot\Delta^{-18/11}\;\le\;2\zeta\;=:\;\mu
\]
for $\Delta$ large. We verify that $(Y_v)_{v\in B}$ is NA, since the variables are neither independent nor obtained from a single permutation. Fix $u\in S$. The vector $\bigl(\mathbf 1_{v\notin N(u)}\bigr)_{v\in B}$ is obtained by placing $b-\Delta'$ ones and $\Delta'$ zeros in a uniformly random order, so it is a permutation distribution and is NA by (a). These $|S|$ vectors, one for each $u\in S$, are independent of one another because the neighborhoods are chosen independently, so by (b) the whole family $\bigl(\mathbf 1_{v\notin N(u)}\bigr)_{u\in S,\,v\in B}$ is NA. Index that family by the $|S|\times|B|$ pairs $(u,v)$ and group it by the second coordinate: the blocks $\{(u,v):u\in S\}$, one for each $v\in B$, are pairwise disjoint, and $Y_v=\prod_{u\in S}\mathbf 1_{v\notin N(u)}$ is a nondecreasing function of the block belonging to $v$. Hence $(Y_v)_{v\in B}$ is NA by (c). Fact (c) is essential here: the $Y_v$ are dependent, since a vertex of $B$ that escapes one neighborhood makes it likelier that another does. By \eqref{eq:chernoff}, applicable as $x_0=\ceil{\sqrt\Delta}>2\zeta=\mu$ for $\Delta$ large,
\[
\Prb{\textstyle\sum_v Y_v\ge x_0}\;\le\;\Bigl(\frac{e\mu}{x_0}\Bigr)^{x_0}
\;\le\;\Bigl(\frac{2e\zeta}{\sqrt\Delta}\Bigr)^{x_0}
\;=\;\exp\Bigl(-x_0\ln\frac{\sqrt\Delta}{2e\zeta}\Bigr)
\;\le\;\exp\Bigl(-\tfrac13\sqrt\Delta\,\ln\Delta\Bigr),
\]
using $\ln\bigl(\sqrt\Delta/(2e\zeta)\bigr)\ge\tfrac13\ln\Delta$ for $\Delta$ large in terms of $\zeta$. The number of pairs $(j,S)$ is at most $(\ell-1)a^{s_0}\le\exp\bigl(\ln\ell+s_0\ln a\bigr)\le\exp\bigl(4\zeta\ln^2\Delta\bigr)$ for large $\Delta$, since $s_0\le2\zeta\ln\Delta+1$ and $a\le\Delta$. A union bound gives failure probability at most
\[
\exp\Bigl(4\zeta\ln^2\Delta-\tfrac13\sqrt\Delta\,\ln\Delta\Bigr)\;\longrightarrow\;0 .
\]
The cost of the enlarged $V_j$ is absorbed here: the union bound runs over $\binom{a}{s_0}$ sets rather than $\binom{\sqrt\Delta}{s_0}$, but the slack $x_0$ permitted by \eqref{eq:gap} grows from $\Theta(\ln^2\Delta)$ to $\Theta(\sqrt\Delta)$, and the Chernoff exponent grows faster than the logarithm of the number of sets. This proves \eqref{eq:eventD}; fix any outcome $G$ satisfying it.

Suppose now, for contradiction, that $G$ admits an equitable $k$-coloring for some $k\le K$. For every such $k$,
\begin{equation}\label{eq:qbound}
\floor*{\frac nk}\;\ge\;\frac nK-1\;\ge\;\frac{\zeta\Delta}{\tfrac12\sqrt{(\ell-1)\Delta/\ln\Delta}}-1\;=\;2\zeta\sqrt{\frac{\Delta\ln\Delta}{\ell-1}}-1 ,
\end{equation}
whereas $a\le\tfrac{Ks_0}{\ell-1}+1\le\bigl(\zeta+o(1)\bigr)\sqrt{\Delta\ln\Delta/(\ell-1)}$. Hence for large $\Delta$
\begin{equation}\label{eq:gap}
a+x_0\;\le\;\bigl(\zeta+o(1)\bigr)\sqrt{\frac{\Delta\ln\Delta}{\ell-1}}+\sqrt\Delta+1\;<\;2\zeta\sqrt{\frac{\Delta\ln\Delta}{\ell-1}}-1\;\le\;\floor*{\frac nk} ,
\end{equation}
using \eqref{eq:qbound} and $\ell-1=o(\ln\Delta)$, the latter because $\ell$ is fixed.

Let $C$ be a color class. Since $V_i$ and $V_j$ are completely joined for $i\ne j$ and $C$ is independent, $C$ meets at most one of $V_1,\dots,V_{\ell-1}$; say $C\subseteq V_j\cup B$. If $|C\cap V_j|\ge s_0$, choose $S\subseteq C\cap V_j$ with $|S|=s_0$. Every vertex of $C\cap B$ is nonadjacent to all of $S$, so $|C\cap B|\le x_0$ by \eqref{eq:eventD}, whence $|C|\le a+x_0$; but equitability forces $|C|\ge\floor{n/k}$, contradicting \eqref{eq:gap}. Hence every class contains fewer than $s_0$ vertices of $V_1\cup\dots\cup V_{\ell-1}$, and as these number $(\ell-1)a$,
\[
k\;>\;\frac{(\ell-1)a}{s_0}\;\ge\;\frac{(\ell-1)}{s_0}\cdot\frac{Ks_0}{\ell-1}\;=\;K ,
\]
contradicting $k\le K$. Hence $\chie(G)>K$, and for $\Delta$ large,
\[
\chie(G)\;>\;\floor*{\tfrac12\sqrt{(\ell-1)\Delta/\ln\Delta}}\;\ge\;\frac{1}{3}\sqrt{\frac{(\ell-1)\Delta}{\ln\Delta}} . \qedhere
\]
\end{proof}

\begin{remark}\label{rem:truth}
The choice of $a$ in \eqref{eq:lbparams} extracts from the obstruction all it gives, making the heuristic of Section~\ref{sec:intro} exact; taking $a=\Theta(\sqrt\Delta)$ instead, small enough that $a/s_0=\Theta(\sqrt\Delta/(\zeta\ln\Delta))$, would lose a factor $\Theta(\zeta\sqrt{\ln\Delta})$.

What the obstruction cannot do is close the remaining gap: it is driven by classes each absorbing $\Theta(\zeta\ln\Delta)$ vertices of $A$, and a typical set of $c\,\zeta\ln\Delta$ vertices of $A$, for small $c$, leaves about $\zeta\Delta^{1-c(1-o(1))}$ vertices of $B$ undominated, so such classes are not excluded. For $\zeta>\ell$ the factor $\sqrt{\ln\Delta}$ separating Theorem~\ref{thm:upper} from Theorem~\ref{thm:lower} is removed by Theorem~\ref{thm:multi}; whether it can also be removed for the threshold $\chies$, and for $\chie$ in the little-room regime of Conjecture~\ref{con:smallzeta}, remains open (Section~\ref{sec:concl}).
\end{remark}

\section{Anchors chosen adaptively}\label{sec:multi}

Theorem~\ref{thm:upper} cannot be pushed below $\sqrt\Delta$, and the reason is visible in one line of the proof of Lemma~\ref{lem:anchor}: Phase~2 estimates $|N(R_i)|\le\rho_i\Delta$, an equality only when the anchors of a class share no neighbors at all. Under that estimate a mixed class absorbs only $O(\zeta)$ vertices of the residue, $t=\Theta(k)$ classes absorb $O(\zeta k)$, the demand is $\Theta(n/k)$, and balancing returns $k=\Theta(\sqrt\Delta)$. The anchors, however, may be chosen freely, and in \emph{any} graph a well-chosen set of $\rho$ anchors dominates far fewer than $\rho\Delta$ vertices. This section makes that precise and exploits it: Theorems~\ref{thm:multi} and~\ref{thm:uniform} are proved by rerunning the proof of Lemma~\ref{lem:anchor} with Phase~2 replaced and condition (P2) dropped, Phases~1 and~3 and Lemma~\ref{lem:multinormal} being unchanged.

\begin{lemma}[Averaging anchor lemma]\label{lem:avg}
Let $A'$ and $B'$ be disjoint nonempty vertex sets in a graph, with every vertex of $A'$ having at most $\Delta$ neighbors in $B'$, and put $p:=|A'|$. Then for every integer $\rho$ with $1\le\rho\le p$ there is a set $R\subseteq A'$ with $|R|=\rho$ and
\[
\bigl|B'\setminus N(R)\bigr|\;\ge\;|B'|\Bigl(1-\frac{\tau\Delta}{|B'|}\Bigr)_{\!+}^{\rho},
\qquad\text{where}\quad \tau:=\frac{p}{p-\rho+1}\;\ge\;1 .
\]
\end{lemma}

\begin{proof}
Let $R$ be a uniformly random $\rho$-element subset of $A'$ and put $p':=p-\rho+1$. For $v\in B'$ let $d_v$ be its number of neighbors in $A'$, and set
\[
\varphi(x)\;:=\;\biggl(\frac{(p'-x)_+}{p'}\biggr)^{\rho},\qquad x\ge0 ,
\]
which is convex and nonincreasing, being the composition of the convex nonincreasing map $x\mapsto(p'-x)_+$ with the convex nondecreasing map $u\mapsto(u/p')^{\rho}$ on $[0,\infty)$. We claim $\Prb{v\notin N(R)}\ge\varphi(d_v)$. Indeed $\Prb{v\notin N(R)}=\binom{p-d_v}{\rho}\big/\binom{p}{\rho}$, which vanishes when $d_v>p-\rho$, as does $\varphi(d_v)$ because then $d_v\ge p'$; and when $d_v\le p-\rho$,
\[
\Prb{v\notin N(R)}\;=\;\prod_{i=0}^{\rho-1}\frac{p-d_v-i}{p-i}\;\ge\;\Bigl(\frac{p'-d_v}{p'}\Bigr)^{\rho}
\]
since $x\mapsto1-d_v/x$ is nondecreasing and $p-i\ge p'$ for $0\le i\le\rho-1$. By Jensen's inequality,
\[
\E\,\bigl|B'\setminus N(R)\bigr|\;=\;\sum_{v\in B'}\Prb{v\notin N(R)}\;\ge\;|B'|\,\varphi(\bar d),
\qquad \bar d:=\frac{1}{|B'|}\sum_{v\in B'}d_v .
\]
The edges counted by $\sum_v d_v$ all leave $A'$, so $\sum_{v\in B'}d_v\le p\Delta$ and $\bar d\le p\Delta/|B'|$; as $\varphi$ is nonincreasing, $\varphi(\bar d)\ge\bigl(1-p\Delta/(p'|B'|)\bigr)_+^{\rho}=\bigl(1-\tau\Delta/|B'|\bigr)_+^{\rho}$. Some outcome attains at least the mean.
\end{proof}

Sampling without replacement makes $|R|=\rho$ exactly, which the counting in the proof of Theorem~\ref{thm:multi} needs: each anchor set must remove a prescribed number of vertices from the residue. The price is the factor $\tau=1+O(\rho/p)$.

\begin{lemma}[Derandomization]\label{lem:derand}
The set $R$ of Lemma~\textup{\ref{lem:avg}} can be found deterministically in time $O\bigl(\rho\,(|A'|\Delta+\rho|B'|)\bigr)$.
\end{lemma}

\begin{proof}
Apply the method of conditional expectations, choosing the elements $u_1,\dots,u_\rho$ of $R$ one at a time, $u_{i+1}$ being taken from $A_i:=A'\setminus\{u_1,\dots,u_i\}$. Write $B_i:=B'\setminus N(\{u_1,\dots,u_i\})$ and $d^{(i)}_v:=|N(v)\cap A_i|$ for $v\in B_i$, and put
\[
\Phi_i(u_1,\dots,u_i)\;:=\;\sum_{v\in B_i}\binom{|A_i|-d^{(i)}_v}{\rho-i}\Big/\binom{|A_i|}{\rho-i} .
\]
If $\{u_{i+1},\dots,u_\rho\}$ is a uniformly random $(\rho-i)$-subset of $A_i$, then $\Phi_i$ is exactly the conditional expectation of $|B'\setminus N(R)|$ given the first $i$ choices; in particular $\Phi_0=\E\,|B'\setminus N(R)|$ and $\Phi_\rho=|B'\setminus N(R)|$, and
\[
\Phi_i(u_1,\dots,u_i)\;=\;\frac{1}{|A_i|}\sum_{u\in A_i}\Phi_{i+1}(u_1,\dots,u_i,u).
\]
Hence some $u\in A_i$ satisfies $\Phi_{i+1}\ge\Phi_i$; choosing such a $u$ at each step gives $\Phi_\rho\ge\Phi_0$, which is the conclusion of Lemma~\ref{lem:avg}.

For the running time, note that $d^{(i+1)}_v=d^{(i)}_v$ for every $v\in B_i$ with $u\notin N(v)$. Hence, with
$w_v:=\binom{|A_i|-1-d^{(i)}_v}{\rho-i-1}\big/\binom{|A_i|-1}{\rho-i-1}$, which does not depend on $u$, and $T:=\sum_{v\in B_i}w_v$,
\[
\Phi_{i+1}(u_1,\dots,u_i,u)\;=\;\sum_{v\in B_i\setminus N(u)}w_v\;=\;T-\sum_{v\in N(u)\cap B_i}w_v ,
\]
so maximizing $\Phi_{i+1}$ over $u\in A_i$ amounts to minimizing $\sum_{v\in N(u)\cap B_i}w_v$. Each $w_v$ is a product of at most $\rho$ factors, so computing all the weights and $T$ costs $O(\rho|B'|)$ per step, and evaluating one candidate costs $O(\Delta)$ once each adjacency list has been restricted to $B'$, a preprocessing pass of cost $O\bigl(\sum_{u\in A'}\deg(u)\bigr)$, absorbed, in every application in this paper, by $O(|A'|\Delta)$; so a step costs $O(\rho|B'|+|A'|\Delta)$. There are $\rho$ steps.
\end{proof}

\begin{remark}\label{rem:input}
Theorems~\ref{thm:multi} and~\ref{thm:uniform} are existence statements and use only that a proper $\ell$-coloring exists. Their proofs are constructive \emph{given} such a coloring; since computing $\chi(G)$ is \textup{NP}-hard~\cite{Karp72}, any algorithmic reading of them takes the $\ell$-coloring as part of the input. For $\ell=2$ this is no restriction (Section~\ref{sec:alg}), which is one reason the bipartite case is stated separately.
\end{remark}

\begin{remark}\label{rem:cost}
Summed over the whole construction, the applications of Lemma~\ref{lem:derand} cost $O(M^2\Delta+\rho Mn)$, where $M=\sum_jM_j$ is the total residue: each class performs $\rho$ steps at cost $O(|A'|\Delta+\rho|B'|)\le O(M\Delta+\rho n)$, and there are $O(M/\rho+\ell\ln^2k)$ of them. Since $M\le(\ell-1)q=O(\ell n/k)$ and one may always take $\rho\le M$, this is $O(\ell^2n^2\Delta/k^2+\ell^2n^3/k^2)$, so the coloring of Theorem~\ref{thm:multi} is computable deterministically in time $O(n+m+\ell^2n^3\ln\Delta/\Delta)$. This is polynomial but not linear; obtaining the sharp bound of Theorem~\ref{thm:multi} in linear time remains open, and is the reason Theorem~\ref{thm:upper}, whose anchors are prescribed, is retained for Theorem~\ref{thm:alg}.
\end{remark}

The improvement is quantitative. The crude bound $|B'|-\rho\Delta$ is vacuous once $\rho\ge|B'|/\Delta$, whereas the conclusion of Lemma~\ref{lem:avg} is at least $q$ for every
\begin{equation}\label{eq:rhochoice}
\rho\;\le\;\ln\frac{|B'|}{q}\Bigl/\ln\frac{|B'|}{|B'|-\tau\Delta}\Bigr. ,
\end{equation}
and the right-hand side is $\bigl(1+o(1)\bigr)\frac{|B'|}{\tau\Delta}\ln\frac{|B'|}{q}$ when $|B'|/\Delta\to\infty$. With $|B'|=\Theta(n)$, $q=\Theta(n/k)$ and $\tau=1+o(1)$ this permits $\rho=\Theta(\zeta\ln k)$ anchors per class in place of $\Theta(\zeta)$, which is exactly the logarithmic factor at stake.

Two properties of \eqref{eq:rhochoice} are used below. It is the exact consequence of Lemma~\ref{lem:avg}, not the weaker form obtained from $1-x\ge e^{-2x}$; the latter halves $\rho$, costs a factor $\sqrt2$ in $\gamma$, and doubles the hypothesis on the order, which is why we do not use it. And the factor $\tau$ is harmless while the residue is large compared with $\rho$, but not on the last few vertices of a residue; there we revert to the crude estimate $|N(R)|\le|R|\Delta$, at a cost of $O(\ln^2k)$ extra classes per part and nothing else. Anchors are therefore chosen adaptively rather than prescribed, in the manner described at the opening of this section.

\begin{proof}[Proof of Theorem~\textup{\ref{thm:multi}}]
Fix a proper $\ell$-coloring with $V_1,\dots,V_\ell$, $B$, $b$ and $a_j$ as in the Notation, and set $k:=\ceil{\gamma\sqrt{\Delta/\ln\Delta}}$, which determines $q$ and $r$. Since $\gamma>\Gamma(\ell,\zeta)$, we may fix a reservation parameter $w$ and a slack $\eta$ with
\begin{equation}\label{eq:eta}
\frac1\zeta<w<\frac1\ell,\qquad 0<\eta<\tfrac12,\qquad
U_0:=(1-\eta)^2w\zeta>1,
\end{equation}
\begin{equation}\label{eq:etab}
(1-\eta)^4\gamma^2\,w\Bigl(\frac1\ell-w\Bigr)\;>\;2(\ell-1)\,h(U_0),
\end{equation}
and we then put $\theta:=\tfrac1\ell-w$, $t:=\floor{\theta k}$ and $\delta:=1-1/U_0>0$. Such a pair $(w,\eta)$ exists: as $\eta\downarrow0$ condition \eqref{eq:etab} becomes $\gamma^2>2(\ell-1)h(w\zeta)\big/\bigl(w(\tfrac1\ell-w)\bigr)$, and by the definition \eqref{eq:Gamma} of $\Gamma$ some admissible $w$ satisfies it whenever $\gamma>\Gamma(\ell,\zeta)$.

Writing $\omega:=\ceil{\ln k}$, fix $\Delta_0$ so large that for every $\Delta\ge\Delta_0$:
\begin{enumerate}
\item[(U0)] $q\ge1$,\quad $k\le\eta wn$,\quad $\omega\ge1/\eta$,\quad $(1-\eta)wk\ge4/\delta$,\quad and\quad $\bigl(1+\tfrac2\delta\bigr)(\ell-1)(1+\ln k)^2\le\eta t$;
\item[(U1)] $k^2\,\ln\bigl((1-\eta)wk\bigr)\;\ge\;\dfrac{(\ell-1)\,h(U_0)}{(1-\eta)^4\,w\,\theta}\,\Delta$.
\end{enumerate}
Writing $k=\gamma\sqrt{\Delta/\ln\Delta}$ we have $\ln k=(\tfrac12+o(1))\ln\Delta$, so the left side of (U1) is $(\tfrac{\gamma^2}{2}+o(1))\Delta$, and (U1) holds for all large $\Delta$ by \eqref{eq:etab}. The clauses of (U0) hold for large $\Delta$ because $k=\Theta(\sqrt{\Delta/\ln\Delta})$, $t\ge\theta k-1$ and $n\ge\zeta\Delta$. Below, $\Delta_0$ is enlarged, finitely often, so that the estimates flagged ``for large $\Delta$'' hold as well; each involves only $\ell$, $\zeta$, $\gamma$, $w$ and $\eta$.

\emph{Splitting the parts.} Since $b\ge n/\ell$, $tq\le\theta k\cdot\tfrac nk=\theta n$ and $r<k\le\eta wn\le wn$, we get $r+tq\le(\theta+w)n=\tfrac n\ell\le b$, so \eqref{eq:Hcond} holds and Lemma~\ref{lem:multinormal} supplies $u_j,v_j,M_j$ and $y\ge0$ with $M_j\le q$, $\sum_j(u_j+v_j)+t+y=k$ and $0\le r-\sum_ju_j\le t+y$. Give part $j$ its $u_j$ classes of size $q+1$ and $v_j$ of size $q$, all inside $V_j$; what remains of $V_j$ is a residue of $M_j\le q$ vertices. Exactly $r-\sum_ju_j$ of the $t+y$ classes still to be built have size $q+1$; assign those sizes arbitrarily.

\emph{The reservoir.} Write $\beta:=b-t(q+1)$. As long as at most $t$ of the remaining classes have been built, at least $\beta$ vertices of $B$ are unused, since each class takes at most $q+1$ of them. Because $b\ge n/\ell$, $t(q+1)\le\theta n+\theta k$ and $k\le\eta wn$,
\begin{equation}\label{eq:beta}
\beta\;\ge\;\frac n\ell-\theta n-\theta k\;\ge\;wn-\eta wn\;=\;(1-\eta)wn\;\ge\;(1-\eta)w\zeta\Delta\;=\;\frac{U_0\,\Delta}{1-\eta}\;>\;\frac{\Delta}{1-\eta},
\end{equation}
using $\theta\le\tfrac1\ell\le1$ in the second step. Also $q\le n/k$ gives $\beta/q\ge(1-\eta)wk\ge4/\delta$ by (U0). Put
\begin{gather*}
U\;:=\;\frac{(1-\eta)\beta}{\Delta}\;\ge\;U_0,\qquad
\lambda\;:=\;\ln\frac{U}{U-1},\\
\rho\;:=\;\min\Bigl\{\Bigl\lfloor\frac1\lambda\ln\frac\beta q\Bigr\rfloor,\,q\Bigr\},
\qquad
\rho_0\;:=\;\min\Bigl\{\Bigl\lfloor\frac{\beta-q-1}{\Delta}\Bigr\rfloor,\,q\Bigr\} .
\end{gather*}
The caps at $q$ ensure that each anchor set fits inside its class, whose size is $q$ or $q+1$. Here $\rho\ge1$: the first entry of the minimum is at least $1$ because $\lambda\le\ln\frac{U_0}{U_0-1}$ is bounded while $\ln(\beta/q)\to\infty$, and $q\ge1$ by (U0). By \eqref{eq:beta}, $\Delta\le\beta/U_0$, so $\beta-\Delta\ge\delta\beta$; and $q+1\le\delta\beta/2$ by (U0), whence
\begin{equation}\label{eq:rho0}
\Bigl\lfloor\frac{\beta-q-1}{\Delta}\Bigr\rfloor\;\ge\;\frac{\beta-q-1-\Delta}{\Delta}\;\ge\;\frac{\delta\beta}{2\Delta},
\end{equation}
while $\frac{\beta-q-1}{\Delta}\ge\bigl(1-\tfrac\delta2\bigr)\frac\beta\Delta\ge\frac{(1-\delta/2)U_0}{1-\eta}=\frac{U_0+1}{2(1-\eta)}>1$ by \eqref{eq:beta} and $U_0>1$; together with $q\ge1$ this gives $\rho_0\ge1$. Since $\lambda\ge1/U$ we also have $\rho\le U\ln(\beta/q)\le(\beta/\Delta)\ln(\beta/q)$; if $\rho_0=q$ then $\rho/\rho_0\le1$, and otherwise \eqref{eq:rho0} applies, so in either case $\rho/\rho_0\le(2/\delta)\ln(\beta/q)$ for $\Delta$ large.

\emph{Building the mixed classes.} Process the parts in turn. For part $j$, let $A'$ be the as yet unused portion of its residue and $B'$ the set of unused vertices of $B$, so that $|B'|\ge\beta$ as long as at most $t$ classes have been built. While $A'\ne\emptyset$, choose an anchor set $R\subseteq A'$ as follows.
\begin{itemize}
\item If $|A'|\ge\omega\rho$, apply Lemma~\ref{lem:avg} to $A'$ and $B'$ with parameter $\rho$, obtaining $R$ with $|R|=\rho$. Its correction factor satisfies $\tau=|A'|/(|A'|-\rho+1)\le\omega/(\omega-1)\le1/(1-\eta)$ by (U0), so $\tau\Delta\le\Delta/(1-\eta)\le\beta/U_0<\beta$ and
\[
|B'\setminus N(R)|\;\ge\;|B'|\Bigl(1-\frac{\tau\Delta}{|B'|}\Bigr)^{\rho}\;\ge\;\beta\Bigl(1-\frac{\Delta}{(1-\eta)\beta}\Bigr)^{\rho}\;=\;\beta\Bigl(1-\frac1U\Bigr)^{\rho}\;=\;\beta\,e^{-\rho\lambda}\;\ge\;q ,
\]
the second step because $s\mapsto s(1-\tau\Delta/s)^{\rho}$ is nondecreasing for $s>\tau\Delta$ and $|B'|\ge\beta$, and the last by the definition of $\rho$.
\item If $|A'|<\omega\rho$, let $R$ consist of any $\min\{\rho_0,|A'|\}$ vertices of $A'$. Then, crudely, $|B'\setminus N(R)|\ge\beta-\rho_0\Delta\ge q+1$.
\end{itemize}
In either case $1\le|R|\le q\le s_i$, so the class has room for its anchors; take $R$ as the anchor set of the next class and complete it with $s_i-|R|\le q$ vertices of $B'\setminus N(R)$. The class is independent: $R$ lies inside the single part $V_j$, $B$ is independent, and no edge joins $R$ to $B'\setminus N(R)$. Delete the vertices used and repeat.

\emph{Counting the classes.} A step of the first kind removes exactly $\rho$ vertices of the residue, so there are at most $M_j/\rho$ of them; a step of the second kind removes $\min\{\rho_0,|A'|\}$ vertices from a residue of size less than $\omega\rho$, so there are at most $\omega\rho/\rho_0+1\le(2\omega/\delta)\ln(\beta/q)+1$ of them. Since $\beta\le n$ and $q\ge n/k-1$ we have $\ln(\beta/q)\le\ln k+1$ for large $\Delta$, so, using $\omega\le\ln k+1$, part $j$ consumes at most $M_j/\rho+(2/\delta)(1+\ln k)^2+1$ classes, and the total is at most
\[
\sum_{j<\ell}\Bigl(\frac{M_j}{\rho}+\frac2\delta(1+\ln k)^2+1\Bigr)
\;\le\;\frac{(\ell-1)q}{\rho}+\Bigl(1+\frac2\delta\Bigr)(\ell-1)(1+\ln k)^2
\;\le\;\frac{(\ell-1)q}{\rho}+\eta t
\]
by (U0) and $M_j\le q$. It therefore suffices that $(\ell-1)q/\rho\le(1-\eta)t$. If $\rho=q$ this reads $\ell-1\le(1-\eta)t$, which holds since $\ell-1\le\eta t\le(1-\eta)t$ by (U0) and $\eta<\tfrac12$. Otherwise $\rho=\bigl\lfloor\tfrac1\lambda\ln\tfrac\beta q\bigr\rfloor$; now $1/\lambda=U/h(U)$ and $h$ is nonincreasing with $U\ge U_0$, so by \eqref{eq:beta},
\[
\rho\;\ge\;\frac{U}{h(U)}\ln\frac\beta q-1\;\ge\;\frac{(1-\eta)\beta}{\Delta\,h(U_0)}\ln\bigl((1-\eta)wk\bigr)-1
\;\ge\;\frac{(1-\eta)^2wn}{\Delta\,h(U_0)}\ln\bigl((1-\eta)wk\bigr)-1 ,
\]
so, using $q\le n/k$ and absorbing the $-1$ for large $\Delta$,
\[
\frac{(\ell-1)q}{\rho}\;\le\;\bigl(1+o(1)\bigr)\,
\frac{(\ell-1)\,h(U_0)\,\Delta}{(1-\eta)^2\,w\,k\,\ln\bigl((1-\eta)wk\bigr)} ,
\]
in which $n$ has canceled. Since $t\ge\theta k-1$, the requirement $(\ell-1)q/\rho\le(1-\eta)t$ reduces to
\[
k^2\,\ln\bigl((1-\eta)wk\bigr)\;\ge\;\bigl(1+o(1)\bigr)\,\frac{(\ell-1)\,h(U_0)}{(1-\eta)^3\,w\,\theta}\,\Delta ,
\]
which is implied by (U1). At most $t$ mixed classes are used, so the bound $|B'|\ge\beta$ was valid throughout. Finally the $y$ remaining classes, together with any of the $t$ reserved ones left unused, are filled from the leftover vertices of $B$, which are independent and, by the computation in Phase~3 of Lemma~\ref{lem:anchor}, exactly the right number. The result is an equitable $k$-coloring.
\end{proof}

Tracking the dependence on $\ell$ in the proof just given yields Theorem~\ref{thm:uniform}.

\begin{proof}[Proof of Theorem~\textup{\ref{thm:uniform}}]
Apply the proof of Theorem~\ref{thm:multi} with
\[
\zeta:=3\ell,\qquad \gamma:=5\ell^{3/2},\qquad w:=\frac{3}{5\ell},\qquad \theta=\frac1\ell-w=\frac{2}{5\ell},\qquad \eta:=\frac1{20},
\]
so that $k=\ceil{5\ell^{3/2}\sqrt{\Delta/\ln\Delta}}$. All the quantities that enter are then independent of $\ell$ or scale with a fixed power of $\ell$, and we check that every condition holds uniformly.

For \eqref{eq:eta}: $1/\zeta=1/(3\ell)<w<1/\ell$, and
\[
U_0=(1-\eta)^2w\zeta=\tfrac{361}{400}\cdot\tfrac95=1.6245>1,
\]
independent of $\ell$; consequently $h(U_0)=1.5530\ldots$ and $\delta=1-1/U_0=0.3844\ldots$ are absolute constants. For \eqref{eq:etab},
\[
(1-\eta)^4\gamma^2w\theta=\tfrac{130321}{160000}\cdot25\ell^3\cdot\tfrac{6}{25\ell^2}=4.887\ldots\,\ell
\;>\;3.106\ldots(\ell-1)=2(\ell-1)h(U_0),
\]
again uniformly in $\ell$.

For (U0): since $\ell\le(\Delta/\ln\Delta)^{1/3}$ we have $k\le5\ell^{3/2}\sqrt{\Delta/\ln\Delta}+1\le6\Delta/\ln\Delta$, while $\eta wn\ge\tfrac1{20}\cdot\tfrac3{5\ell}\cdot3\ell\Delta=0.09\,\Delta$, so $k\le\eta wn$ once $\ln\Delta\ge67$. The clauses $q\ge1$ and $\omega\ge20$ hold for large $\Delta$, and $(1-\eta)wk\ge\tfrac{19}{20}\cdot\tfrac3{5\ell}\cdot5\ell^{3/2}\sqrt{\Delta/\ln\Delta}=2.85\sqrt\ell\sqrt{\Delta/\ln\Delta}\ge4/\delta$ for large $\Delta$. For the last clause it suffices, since $t\ge\theta k-1$, that $C\ell^2(1+\ln k)^2\le k$ for an absolute constant $C$; and $\ell^2\le\ell^{3/2}(\Delta/\ln\Delta)^{1/6}$ together with $k\ge5\ell^{3/2}\sqrt{\Delta/\ln\Delta}$ reduces this to $(1+\ln k)^2\le(5/C)(\Delta/\ln\Delta)^{1/3}$, which holds for large $\Delta$ because $\ln k=O(\ln\Delta)$, uniformly in $\ell$. This is the clause that binds, and it permits $\ell$ far beyond $(\Delta/\ln\Delta)^{1/3}$; that hypothesis is imposed for definiteness, and because the bound exceeds $\Delta$ not far above that threshold (Section~\ref{sec:intro}).

For (U1) only a lower bound on the logarithm is needed. Since $(1-\eta)wk\ge2.85\sqrt\ell\sqrt{\Delta/\ln\Delta}$, which is at least $2\sqrt{\Delta/\ln\Delta}$, we get $\ln((1-\eta)wk)\ge(\tfrac12-o(1))\ln\Delta$ with the $o(1)$ absolute, so
\[
k^2\ln\bigl((1-\eta)wk\bigr)\;\ge\;\bigl(1-o(1)\bigr)\,25\,\ell^3\,\frac{\Delta}{\ln\Delta}\cdot\frac{\ln\Delta}{2}\;=\;\bigl(\tfrac{25}2-o(1)\bigr)\ell^3\Delta ,
\]
whereas the right-hand side of (U1) equals $(\ell-1)h(U_0)\Delta\big/\bigl((1-\eta)^4w\theta\bigr)=7.9446\ldots\,\ell^2(\ell-1)\Delta\le7.945\,\ell^3\Delta$. As $\tfrac{25}2=12.5>7.945$, condition (U1) holds for all large $\Delta$, uniformly in $\ell$. The finitely many further ``for large $\Delta$'' absorptions in the proof of Theorem~\ref{thm:multi} involve only the absolute constants $U_0$, $h(U_0)$, $\delta$, $\eta$ and $w\zeta=\tfrac95$, together with the bounds $k\ge5\cdot2^{3/2}\sqrt{\Delta/\ln\Delta}$ and $k/n\le1/\ln\Delta$, and are therefore uniform in $\ell$ as well. A single $\Delta_1$ therefore works.
\end{proof}

\begin{remark}\label{rem:adaptive}
Two features of Theorem~\ref{thm:upper} must be reconsidered when the anchors are found rather than prescribed. The first is the linear-time implementation of Section~\ref{sec:alg}, since Lemma~\ref{lem:avg} guarantees only that a good anchor set \emph{exists}; sampling $O(k)$ candidate sets and keeping the best succeeds with high probability, by the reverse Markov inequality applied to $|B'\setminus N(R)|\le|B'|$, so Theorem~\ref{thm:multi} is realized by a randomized polynomial-time algorithm; Lemma~\ref{lem:derand} makes it deterministic, at the cost recorded in Remark~\ref{rem:cost}, and only the linear-time implementation is lost. The second is only partly retained. The construction of Theorem~\ref{thm:multi} applies unchanged, one $k$ at a time, to every $k$ with $k_1\le k\le\eta wn$: the binding condition \textup{(U1)} only improves as $k$ grows, and so, once $k_1$ is large, do the clauses of \textup{(U0)} other than $k\le\eta wn$. That last clause caps the range, so the full threshold form does not follow here as it did from Lemma~\ref{lem:mono}; for $\zeta>2\ell$ the range beyond $\eta wn$ is in any case covered by Theorem~\ref{thm:upper} and the Hajnal--Szemer\'edi theorem, at the cost of a factor $\sqrt{\ln\Delta}$ in the bound, while for $\ell<\zeta\le2\ell$ we make no claim about $\chies$. In the other direction, neither Theorem~\ref{thm:multi} nor Theorem~\ref{thm:upper} subsumes Theorem~\ref{thm:smallzeta}, whose range $\zeta>1$ includes bipartite graphs with $1<\zeta\le2$, where no other bound of this paper applies.
\end{remark}

\begin{remark}\label{rem:delta0}
The two constructions differ sharply in the size of $\Delta_0$. That of Theorem~\ref{thm:upper} is small and explicit: conditions (C0)--(C2) are polynomial inequalities in $k_1$ alone, and Appendix~\ref{app:exp} evaluates them. The $\Delta_0$ of Theorem~\ref{thm:multi} is not of that kind. Condition (U0) requires $\ceil{\ln k}\ge1/\eta$ and $k\le\eta wn$, and $\eta$ must be taken small when $\zeta$ or $\gamma$ is near critical, so $\Delta_0$ grows exponentially in $1/\eta$; for the parameters of Theorem~\ref{thm:uniform}, where $\eta=1/20$, the second clause alone requires $\ln\Delta\ge67$. We have made no attempt to optimize this, since the interest of Theorem~\ref{thm:multi} is the growth rate rather than the threshold; the small thresholds of Appendix~\ref{app:exp} do not transfer to the sharp bound.
\end{remark}

\begin{remark}\label{rem:sharp}
The hypothesis of bounded chromatic number cannot be dropped. The disjoint union of $K_{\Delta+1}$ with $(\zeta-1)\Delta$ isolated vertices has maximum degree $\Delta$ and order at least $\zeta\Delta$, yet $\chie\ge\chi=\Delta+1$. Some order condition also remains necessary for every $\ell$: the complete multipartite graph $K_{1,\dots,1,s}$ with $\ell-1$ singleton parts has $n=s+\ell-1$ and $\Delta(G)=n-1$ for fixed $\ell$; its $\ell-1$ universal vertices must form singleton classes, equitability then forces all classes to have size at most $2$, and $\chie\ge\ceil{n/2}=\bigl(\tfrac12+o(1)\bigr)\Delta(G)$. For $\ell=2$ this is the star.
\end{remark}

\section{A linear-time algorithm}\label{sec:alg}

We now prove Theorem~\ref{thm:alg}. Assume $G$ is given by adjacency lists together with a proper $\ell$-coloring; for $\ell=2$ the latter is free, a bipartition being computable by breadth-first search in $O(n+m)$ time, and isolated vertices may be placed on either side. Identifying a largest part $B$ costs $O(n)$. The numbers $k,q,r,t,\beta,L,H$, the greedy pass of Lemma~\ref{lem:multinormal} --- which performs $O(k)$ arithmetic steps in total, since it allocates at most $k$ classes --- and the balanced anchor sizes of Observation~\ref{obs:distribute} are computed in $O(k+\ell)\subseteq O(n)$ operations; checking (A0)--(A3) is $O(1)$. It remains to implement the three phases of Lemma~\ref{lem:anchor} in linear time. Phases~1 and~3 partition subsets of the parts $V_j$ and of $B$ into consecutive blocks of prescribed sizes, which takes $O(n)$. For Phase~2, maintain the unused vertices of $B$ in a doubly linked list, together with an array $\mathrm{stamp}[\,\cdot\,]$ indexed by $V(G)$, initialized to $0$; entries of vertices outside $B$ may be written but are never consulted. When mixed class $D_i$ ($i$-th in the phase, $i=1,\dots,t$) is built: first select its anchor set $R_i$ (the next $\rho_i$ vertices of the residue of the part $V_{j(i)}$ it is labeled by), and set $\mathrm{stamp}[z]\leftarrow i$ for every $z\in N(R_i)$ by scanning the adjacency lists of $R_i$; then walk the linked list from its head, skipping vertices with $\mathrm{stamp}=i$ and moving vertices with $\mathrm{stamp}\ne i$ from the list into $D_i$, until $s_i-\rho_i$ vertices have been collected. Lemma~\ref{lem:anchor} guarantees the walk succeeds.

The total cost of stamping is $\sum_{i=1}^{t}\sum_{v\in R_i}\deg(v)\le 2m$ (at most $m$ for bipartite $G$, the anchors all lying in one part), because the anchor sets $R_i$ are pairwise disjoint. Each step of a list walk either deletes a vertex from the list --- which happens at most $n$ times in total, as deleted vertices never return --- or skips a vertex $z$ with $\mathrm{stamp}[z]=i$. In the latter case $z\in N(R_i)$, so there is an edge from $R_i$ to $z$; since each walk passes each list position at most once and the sets $R_i$ are pairwise disjoint, distinct skips are charged to distinct edges of $G$. Hence the walks cost $O(n+m)$ overall, and the whole algorithm runs in time $O(n+m)$.

The implementation applies verbatim for every $k$ with $\ceil{\gamma\sqrt\Delta}\le k\le\Delta$, which is exactly the range covered in Step~2 of the proof of Theorem~\ref{thm:upper}; the reserved number $t$ is computed once, from $k_1$, and does not change with $k$. For $k\ge\Delta+1$ the graph is equitably $k$-colorable by the Hajnal--Szemer\'edi theorem, and such a coloring is computable in polynomial time~\cite{KierKost08,KKMS10}, though not within the linear bound above; in that range, however, the number of colors is already at least $\Delta+1$, so the bound of Theorem~\ref{thm:upper} carries no information. \qed

\section{Concluding remarks}\label{sec:concl}

We conclude by highlighting three directions left open by the present results.

The first is the hypothesis on the order. For $\chie$ on bipartite graphs the threshold is settled: it is exactly $n>\Delta$, by Theorem~\ref{thm:smallzeta} and the star. What Theorem~\ref{thm:smallzeta} loses is the logarithm: it gives $O_\zeta(\sqrt\Delta)$ where Theorem~\ref{thm:multi} gives $O(\sqrt{\Delta/\ln\Delta})$, and only for $\zeta>\ell$. The gap between the two motivates the following conjecture.

\begin{conjecture}\label{con:smallzeta}
For every $\zeta>1$ there is $C(\zeta)$ such that every bipartite graph $G$ with $n\ge\zeta\Delta$ and $\Delta$ sufficiently large satisfies $\chie(G)\le C(\zeta)\sqrt{\Delta/\ln\Delta}$.
\end{conjecture}

The origin of the requirement $\zeta>\ell$ in Theorem~\ref{thm:multi} is as follows (cf.\ Remark~\ref{rem:special2}), and it indicates what an attack on the conjecture must supply. The construction reserves $t\approx\theta k$ mixed classes, hence about $\theta n$ vertices of $B$, and bounds $b$ below only by $n/\ell$, so the reservoir available to the mixed classes is about $(\tfrac1\ell-\theta)n$. This must exceed $\Delta$, since a single anchor already forbids $\Delta$ vertices of $B$; letting $\theta\downarrow0$ therefore forces $\zeta>\ell$, and $\Gamma(\ell,\zeta)$ diverges as $\zeta\downarrow\ell$ because $\theta$ must then be taken close to $0$. The prescribed-anchor construction of Theorem~\ref{thm:upper} loses a further factor $2$ to the floor in $L=\floor{\beta/(\Delta-1)}$.

The bound $b\ge n/\ell$ is the natural point of attack. It is tight only when all $\ell$ parts have the same size, whereas the residues that must be absorbed are largest when they do not; a preliminary partition of $B$, or a construction that reuses vertices of $B$ between mixed classes, appears to us the natural route. What cannot help is a sharper form of the averaging lemma: \eqref{eq:rhochoice} already uses it exactly, and the requirement that the reservoir exceed $\Delta$ survives any improvement to it. We note also that for the threshold $\chies$ no such conjecture is available: by Proposition~\ref{prop:threshsharp} the order hypothesis for $\chies$ must exceed $2$, and we do not know the truth between $2$ and the $4$ of Theorem~\ref{thm:upper}.

The second is the constant factor that separates Theorems~\ref{thm:multi} and~\ref{thm:lower}: the lower constant is independent of $\zeta$, and the upper constant is bounded uniformly once $\zeta/\ell$ is bounded away from $1$. Determining the extremal function to within $1+o(1)$ would require work at both ends. Remark~\ref{rem:truth} shows that the domination obstruction of Section~\ref{sec:lower} is already tight for the construction that realizes it: it caps a class at $\Theta(\zeta\ln\Delta)$ vertices of one small part, and Theorem~\ref{thm:lower} takes those parts as large as equitability permits. Closing the gap therefore requires either a different obstruction, or an upper-bound argument in which one class absorbs unboundedly many vertices of a small part at a cost below $\Delta$ per anchor, which the anchor counting used here cannot achieve.

The third is the dependence on $\ell$, where the two bounds bracket the truth but do not meet: Theorem~\ref{thm:multi} gives $O(\ell^{3/2})$ and Theorem~\ref{thm:lower} gives $\Omega(\ell^{1/2})$. Both exponents are artifacts of their arguments. In the upper bound one factor $\ell$ comes from the reservation $t\approx\theta k$ with $\theta=\Theta(1/\ell)$ and one factor $\sqrt{\ell-1}$ from the $\ell-1$ residues to be absorbed. In the lower bound the factor $\sqrt{\ell-1}$ comes from making the $\ell-1$ small parts pairwise complete, so that no class meets two of them, and then counting: each part contributes $|V_j|/s_0$ classes and no class is counted twice. Neither mechanism appears canonical, and we regard both endpoints as plausible. Deciding between $\ell^{1/2}$ and $\ell^{3/2}$ is, in our view, the most interesting question the paper raises. Any attack on it must remove one of the two rigidities just described --- through a construction whose mixed classes draw their anchors from more than one part at a time, or through an obstruction that couples the small parts by less than complete joins. A related unknown is whether $\chie$ is linear in $\Delta$ for some family with $\chi(G)$ just above the threshold of Theorem~\ref{thm:uniform}: the examples with $\chie=\Omega(\Delta)$ that satisfy the order hypothesis, such as the disjoint union of Remark~\ref{rem:sharp}, all have $\chi(G)=\Omega(\Delta)$.

Two further questions are of a different character. The list analogue is open: is every bipartite graph with $n\ge\zeta\Delta$ equitably $k$-choosable for $k=O_\zeta(\sqrt\Delta)$? The construction here uses the global structure of both sides --- which classes lie in $A$, which in $B$, and how the residue is distributed --- and it is not clear how to adapt any of that to lists. And on the algorithmic side, Theorem~\ref{thm:alg} shows that equitable coloring is solvable in linear time in the regime studied here, in contrast with the general problem; the complexity of computing $\chie$ exactly on large bipartite graphs remains open, as does whether the anchor sets of Lemma~\ref{lem:avg} can be found in linear time. The corresponding search problem for balanced colorings is limited by a barrier of the kind familiar from random constraint satisfaction~\cite{Chakraborti23}; whether anything similar applies here we do not know.

\appendix

\section{Numerical experiments}\label{app:exp}

The algorithm of Section~\ref{sec:alg}, the multipartite construction of Section~\ref{sec:multi}, and the random construction of Section~\ref{sec:lower} were implemented. For every coloring produced, it was checked that all classes are independent, that they partition $V(G)$, that class sizes differ by at most one, and that exactly $r$ classes have size $q+1$. All results of this paper are proved in full in the preceding sections; nothing below is used in any proof, and the computations serve only as independent confidence checks on the constructions and the constants. Representative outcomes follow.

\subsection*{The prescribed-anchor construction} The construction of Section~\ref{sec:upper} was run through a single code path for all $\ell$. For $\ell=2$ the inputs were four families of bipartite graphs --- $K_{\Delta,\Delta}$ together with isolated vertices, random $\Delta$-regular bipartite graphs, disjoint unions of stars $K_{1,\Delta}$, and the random instance of Section~\ref{sec:lower} itself --- with $n=\ceil{\zeta\Delta}$, for $\Delta\in\{256,400,900\}$ and
\[
(\zeta,\gamma)\in\bigl\{(25,9),\,(30,8),\,(50,6),\,(15,12),\,(25,5)\bigr\},
\]
at $k=\ceil{\gamma\sqrt\Delta}$. For $\ell\in\{3,4,5\}$ the inputs were random $\ell$-partite graphs with $\ell-1$ small parts and maximum degree $\Delta$, for $\Delta\in\{256,400\}$ and $\zeta\in\{25,50\}$, with $\gamma$ exceeding the critical value of Theorem~\ref{thm:upper} by $15\%$. Every run returned a valid equitable coloring in the sense just described.

The pair $(25,5)$ tests the conditions near their critical value: for $\ell=2$ and $\zeta=25$ the critical value is $4\sqrt{25/21}=4.364\ldots$, and $\gamma=5$ exceeds it by less than $15\%$; every run nevertheless succeeded. For instance, at $\Delta=1024$ and $\gamma=5$ the bound uses $k=160$ colors against the $\ceil{\Delta/2}+1=513$ of the best linear bound, and $\ceil{5\sqrt\Delta}\le\ceil{\Delta/2}+1$ for every $\Delta\ge95$.

\subsection*{Threshold range} On the bipartite families above, at $\Delta\in\{256,400\}$ and $\zeta\in\{10,25\}$, the construction with $t$ frozen at $\floor{k_1/(2\ell)}$ succeeded for \emph{every} $k$ with $k_1=\ceil{9\sqrt\Delta}\le k\le\Delta$, as Theorem~\ref{thm:upper} predicts. Conditions (A0)--(A3) were also evaluated directly across that interval, confirming the monotonicity of Lemma~\ref{lem:mono}: $q$ is nonincreasing, $\beta$ and $L$ nondecreasing, and $\ceil{q/H}$ nonincreasing in $k$ in every instance.

\subsection*{Normalized forms} The conclusions of Lemma~\ref{lem:multinormal} were verified exhaustively for $\ell=2$ on all $3.07\times10^{5}$ tuples $(k,q,r,t,a_1,b)$ with $2\le k<16$, $1\le q<10$, $0\le r<k$, $1\le t\le k$ and $a_1+b=n$ satisfying \eqref{eq:Hcond}, and checked on $3\times10^{5}$ random parameter tuples $(\ell,k,q,r,t,a_1,\dots,a_{\ell-1},b)$ with $\ell\le7$, $k\le300$ and $q\le500$ satisfying \eqref{eq:Hcond}: the bounds $M_j\le q$, $\sum_j(u_j+v_j)\le k-t$, $\sum_ju_j\le r$, $\sum_jv_j\le k-r$ and $0\le r-\sum_ju_j\le t+y$ held in every instance, the first with equality attained.

\subsection*{The explicit conditions imply the abstract ones} Conditions (C0)--(C2) of Theorem~\ref{thm:upper} are stated for $k_1$ alone, while Proposition~\ref{prop:core} asks for (A0)--(A3), which involve the part sizes as well. The implication was checked directly at the worst admissible part sizes, namely $b=\ceil{n/\ell}$ and $n=\ceil{\zeta\Delta}$: over $3840$ parameter points with $\ell\in\{2,3,4,5\}$, $\zeta$ ranging from $2\ell+0.05$ to $10\ell$, $\gamma$ from $1.02$ to $3$ times its critical value, and $\Delta$ swept geometrically over more than ten orders of magnitude from the smallest value at which (C0)--(C2) hold, conditions (A0)--(A3) held at every point. In particular they hold throughout $2\ell<\zeta<4\ell$, the range in which the constant $2\ell$ of Theorem~\ref{thm:upper} is binding.

\subsection*{Little room} Lemma~\ref{lem:arith} was verified exhaustively on all $1.6\times10^{6}$ pairs $(n,a)$ with $n<2500$ and $a\le n/2$ for which $q:=\floor{\sqrt{\min\{a,n/5\}}}$ satisfies its hypotheses: in every case the prescribed $k=\floor{n/q}-1$ admits a split, with no failures. The case division in the proof of Theorem~\ref{thm:smallzeta} was then checked directly: for $\zeta\in\{1.2,1.5,2,3\}$ and $\Delta\in\{2000,10^4\}$, every admissible pair of part sizes $a\le b$ with $b\ge\Delta$ falls into Case~1 or Case~2, and the resulting number of colors never exceeded $1.06\,C(\zeta)\sqrt\Delta$.

\subsection*{Adaptive anchoring} Lemma~\ref{lem:avg} was tested against its own guarantee on three families of bipartite graphs --- neighborhoods uniformly random, concentrated in blocks, and translates of a fixed interval --- for $\Delta\in\{2000,4000\}$ and $\zeta\in\{8,20\}$, at a value of $\rho$ larger than the proof uses: in every case the best of $300$ sampled $\rho$-subsets left more than the guaranteed $|B'|(1-\tau\Delta/|B'|)^{\rho}$ vertices of $B'$ free. The margin indicates the following. For uniformly random neighborhoods the degrees $d_v$ are nearly equal, so Jensen's inequality is almost tight and the best sample exceeded the guarantee by only $16\%$ to $26\%$; for the two structured models it exceeded it by factors between $3.7$ and $10.8$. The guarantee is thus close to best possible in general, and far from it on any instance with unequal degrees. The construction of Theorem~\ref{thm:multi} was then run for
\[
(\Delta,\zeta)\in\bigl\{(2\cdot10^4,12),\,(5\cdot10^4,12),\,(5\cdot10^4,25)\bigr\}
\]
at $k=\ceil{\gamma\sqrt{\Delta/\ln\Delta}}$ with $\gamma\in\{11,13\}$, parameters satisfying \textup{(U1)}. Each run placed the whole residue using between $23\%$ and $36\%$ of the $t$ available mixed classes, and in four of the six the resulting $k$ was below the $\ceil{4\sqrt\Delta}$ that a prescribed-anchor construction would need.

\subsection*{Lower bound} The parameters \eqref{eq:lbparams} were instantiated directly. For $\ell=2$, $\Delta=2500$ and $\zeta=5$ one gets $s_0=79$, $K=8$, $a=632$, $b=12\,500$, $x_0=50$ and $\floor{n/K}=1641$, so the requirement $a+x_0<\floor{n/K}$ of \eqref{eq:gap} holds with a wide margin; over $200$ sampled sets $S\in\binom{V_1}{s_0}$ the largest number of undominated vertices of $B$ was $0$, well below $x_0$. The same held for $(\Delta,\zeta)\in\{(2500,2),(10^4,2),(10^4,5)\}$, and in every case the maximum degree of $B$ was far below $\Delta$, confirming $\Delta(G)=\Delta$. For $\ell\in\{3,5,10\}$ and $\Delta$ ranging over $10^{6},10^{8},10^{10}$, the four structural requirements of \eqref{eq:lbparams} --- $(\ell-1)a\le\Delta$, $0<\Delta'\le b$, $a+x_0<\floor{n/K}$ and $(\ell-1)a\ge Ks_0$ --- were checked, and all hold; the domination sampling was repeated for $\ell\in\{3,5\}$ at $\Delta\in\{2500,10^4\}$ with the same outcome. These are samples, not certificates: the proof quantifies over all $\binom{a}{s_0}$ sets, which the union bound in Section~\ref{sec:lower} handles. It was also verified numerically that each inequality behind that union bound holds, namely $a\le\Delta$, $b\ge\Delta$, $\mu\le2\zeta$, $a+x_0<\floor{n/K}$ and $s_0\ln a<x_0\ln\bigl(x_0/(e\mu)\bigr)$, for $\Delta$ ranging over $10^{4},\dots,10^{12}$ and $\zeta$ over $\{2,5,25,100\}$. The last is the binding one: it holds for all four values of $\zeta$ once $\Delta\ge10^{10}$, and fails for $\zeta=100$ at $\Delta=10^{8}$, so the requirement that $\Delta$ be large in terms of $\zeta$ is polynomial rather than iterated-logarithmic, a consequence of the enlarged $x_0$.

\subsection*{Explicit thresholds} The smallest $k_0$ from which the explicit sufficient conditions (C0)--(C2) hold, and the resulting $\Delta_0\approx(k_0/\gamma)^2$, with $\gamma$ taken $15\%$ above the critical value $\gamma_{\mathrm{crit}}=2\ell\sqrt{\zeta(\ell-1)/(\zeta-2\ell)}$:

\begin{center}
\begin{tabular}{ccrrr}
\toprule
$\ell$ & $\zeta$ & $\gamma_{\mathrm{crit}}$ & $k_0$ & $\Delta_0$ \\
\midrule
$2$ & $25$ & $4.364$ & $46$ & $84$ \\
$2$ & $50$ & $4.170$ & $39$ & $67$ \\
$3$ & $25$ & $9.733$ & $128$ & $131$ \\
$3$ & $50$ & $9.045$ & $96$ & $86$ \\
$4$ & $50$ & $15.119$ & $188$ & $117$ \\
$5$ & $60$ & $21.909$ & $300$ & $142$ \\
\bottomrule
\end{tabular}
\end{center}

\noindent
All of these thresholds belong to Theorem~\ref{thm:upper}; they say nothing about Theorem~\ref{thm:multi}, whose $\Delta_0$ is of a completely different size (Remark~\ref{rem:delta0}). They are small: for bipartite graphs of order at least $25\Delta$ the construction is certified for every $\Delta\ge84$, and since $t$ is frozen the same certificate covers every $k$ up to $\Delta$ at no extra cost. The construction itself succeeds well before these worst-case guarantees.

\end{document}